\documentclass{article}

\usepackage{comment}
\usepackage[utf8]{inputenc}
\usepackage[T1]{fontenc}
\usepackage{amssymb}
\usepackage{amsmath, amsthm}
\usepackage{amsfonts}
\usepackage{graphics}
\usepackage{color}
\usepackage{xcolor}
\usepackage{graphicx}
\usepackage{graphics}
\usepackage{color}
\usepackage{amsmath}
\usepackage{amsfonts}
\usepackage{amssymb}
\usepackage{enumerate}
\usepackage{times,amsmath,amsthm,amssymb,dsfont,graphicx,xspace,epsfig,xcolor}

\newtheorem{Proposition}{Proposition}

\newtheorem{Theorem}{Theorem}

\newtheorem{Claim}{Claim}
\newtheorem{Lemma}{Lemma}

\newtheorem{case}{Case}
\newtheorem{casefirst}{Case}
\newtheorem{Subcase}{Subcase}[case]
\newtheorem{Problem}{Problem}
\newtheorem{obs}{Observation}
\newtheorem{definition}{Definition}
\begin{document}
\title{Decompositions into two linear forests of bounded lengths}

\author{Rutger Campbell\thanks{Supported by the Institute for Basic Science (IBS-R029-C1).} \\rutger@ibs.re.kr \\Discrete Mathematics Group, Institute for Basic Science (IBS),\\ Daejeon, Republic of Korea \and
Florian Hörsch\\florian.hoersch@tu-ilmenau.de \\TU Ilmenau, Weimarer Straße 25, Ilmenau, Germany, 98693  \and 
Benjamin Moore\thanks{Supported by project 22-17398S (Flows and cycles in graphs on surfaces) of
Czech Science Foundation.} \\ brmoore@iuuk.mff.cuni.cz \\ Institute of Computer Science, Charles University, Prague, Czechia}


\maketitle

\begin{abstract}
For some $k \in \mathbb{Z}_{\geq 0}\cup \infty$, we call a linear forest $k$-bounded if each of its components has at most $k$ edges. 
We will say a $(k,\ell)$-bounded linear forest decomposition of a graph $G$ is a partition of $E(G)$ into the edge sets of two linear forests $F_k,F_\ell$ where $F_k$ is $k$-bounded and $F_\ell$ is $\ell$-bounded. 
We show that the problem of deciding whether a given graph has such a decomposition is NP-complete if both $k$ and $\ell$ are at least $2$, NP-complete if $k\geq 9$ and $\ell =1$, and is in P for $(k,\ell)=(2,1)$. Before this, the only known NP-complete cases were the $(2,2)$ and $(3,3)$ cases. Our hardness result answers a question of Bermond et al. from 1984.
We also show that planar graphs of girth at least nine decompose into a linear forest and a matching, which in particular is stronger than $3$-edge-colouring such graphs.
\end{abstract}

\section{Introduction}
In this paper, all graphs are finite, with no loops but possibly with parallel edges.  We are interested in \textit{graph decomposition} problems. Given a graph $G$, we say that a collection $(H_1,\ldots,H_t)$ of spanning subgraphs of $G$ is a {\it decomposition} of $G$ if $(E(H_1),\ldots,E(H_t))$ is a partition of $E(G)$. 
Decompositions are only interesting if further constraints are imposed, so we will focus on decompositions where various parts are forced to belong to a family of graphs, and possibly the number of parts is restricted. We will be interested in the algorithmic complexity of such questions. 

For example, a natural question is to ask if a graph $G$ decomposes into parts that are each 
isomorphic to a fixed graph $H$. This is NP-complete for almost all graphs $H$, and a dichotomy theorem is known \cite{Hoyer}. On the other hand, one could ask for decompositions into a family of natural graphs, such as forests. Of course, every graph admits a decomposition into forests, so this problem is only interesting if we restrict the number of parts. Even with this restriction, the famous Nash-Williams Theorem \cite{nash} characterizes when a graph decomposes into $k$ forests, and this characterization gives rise to polynomial-time algorithms. 

This result of Nash-Williams raises the question of whether or not decompositions exist when we impose further restrictions on the trees of the decomposition. The literature dedicated to these problems is already rich. For instance, many authors consider decompositions where one tree has bounded degree. Balogh et al \cite{conj24} showed that planar graphs can be decomposed into three forests, where one of these forests has maximum degree $8$. Gonçalves \cite{GONCALVES2009314} later improved their result to show that maximum degree $4$ is achievable. For general graphs, Jiang and Yang proved the famous Nine Dragon Tree Theorem  \cite{ndt}, which gives sharp bounds on when a graph decomposes into $k+1$ forests where one of these forests has maximum degree at most $d$.
Similar lines of research have been conducted when we ask for one of the forests to have bounded component size, and the Strong Nine Dragon Tree conjecture, posed by Montassier et al in \cite{sndtck1d2} remains wide open. Nevertheless, there are some partial results, for example by Mies and the third author, \cite{SNDTisolates}, by Kim et al \cite{k=1d=2}, and by Yang \cite{YANG201840}.

We are interested in a problem where the degree and the component size restrictions are combined and imposed on all forests in the decomposition rather than on a single one. Firstly, we consider linear forests which are defined as vertex-disjoint collections of paths. Observe that these are exactly forests of maximum degree at most 2. Decompositions into linear forests were first considered by Akiyama, Ekoo and Harary \cite{Akiyama1981CoveringAP}. In order to also incorporate the component size restriction, we study {\it $k$-bounded linear forests}, meaning linear forests where each component has at most $k$ edges. Observe that a $1$-bounded linear forest is a matching.

Structural questions on the decompositions of regular graphs into linear forests of bounded length have been considered by Alon, Teague and Wormald \cite{atw}. Additionally, Thomassen showed that every cubic graph decomposes into two $5$-bounded linear forests \cite{THOMASSEN1999100}. From an algorithmic point of view, the above observation means that the problem of decomposing a graph into $t$ linear forests of bounded length for some positive integer $t$ generalizes the problem of decomposing the graph into $t$ matchings.

A moment of thought shows that this is the $t$-edge-colouring problem, which is NP-complete for all $t \geq 3$, see \cite{doi:10.1137/0210055}. Here, recall that a \textit{$t$-edge-colouring} of a graph $G$ is a map $f:E(G) \to \{1,\ldots,k\}$ such that for all incident edges $e_{1}$ and $e_{2}$ we have that $f(e_{1}) \neq f(e_{2})$. Thus decomposing into linear forests is closely related to edge-colouring graphs. Therefore it is no surprise that it is NP-complete to decide if a graph decomposes into $k$-linear forests, for any fixed $k \geq 2$ \cite{treespathsandspiders}. We will be interested in the algorithmic complexity of the decision problem where we only have two parts. More precisely, we consider the following problem class:\medskip

\noindent \textbf{$(k,\ell)$-bounded linear forest decomposition ($(k,\ell)$-BLFD):}
\smallskip

\noindent\textbf{Input:} A graph $G$.
\smallskip

\noindent\textbf{Question:} Does $G$ decompose into a $k$-bounded linear forest and an $\ell$-bounded linear forest?
\medskip

One important case has been settled by Péroche who proved the following result:
\begin{Theorem} [\cite{PEROCHE1984195}]
\label{inftyinfty}
    $(\infty,\infty)$-BLFD is NP-complete.
\end{Theorem}

If $k=1$ and $\ell=1$, then this is simply asking if a graph decomposes into two matchings, or equivalently a $2$-edge-colouring of a graph, which is easily seen to be polynomial-time solvable. Interestingly, to the best of the authors' knowledge, it is open if a graph can decompose into a matching and a $2$-bounded linear forest. Our first result is that this is polynomial-time solvable:

\begin{Theorem}\label{21poly}
$(2,1)$-BLFD is polynomial-time solvable. 
\end{Theorem}

On the negative side, we generalize known hardness results to obtain nearly a full classification. In particular, in \cite{smallsubgraphcomponents} the authors show that $(2,2)$-BLFD is NP-complete and in \cite{BERMOND1984123}, the authors show that $(3,3)$-BLFD is NP-complete. Bermond et al. \cite{BERMOND1984123} conjectured that $(k,k)$-BLFD is NP-complete for all $k \geq 2$. We prove the conjecture, and in fact show this result holds for almost all values of $k,\ell$. More precisely:

\begin{Theorem}\label{mainhard}
For any $k,\ell \in \{2,\ldots\} \cup \{\infty\}$, we have that $(k,\ell)$-BLFD is NP-complete. Additionally, for $\ell=1$, and all $k \in \{9,\ldots\} \cup \infty$, we have that $(k,\ell)$-BLFD is NP-complete. 
\end{Theorem}

This solves the complexity problem except for the cases where $\ell =1$ and $k \in \{3,4,5,6,7,8\}$. 

Note that a decomposition of a graph into a linear forest and a matching gives rise to a $3$-edge-colouring of this graph, as any linear forest can be decomposed into two matchings, thus using these two matchings as colour classes, and the other matching as the third colour class we obtain a $3$-edge-colouring. Further, it is well known that planar triangulations admit a $4$-colouring if and only if their dual graphs admit a $3$-edge-colouring. As there are subcubic planar graphs which are not duals of triangulations, it is interesting to ask which subcubic planar graphs are $3$-edge-colourable. We define the \textit{girth} of a graph as the length of the shortest cycle in the graph. It is open whether every subcubic planar graph with girth at least six is $3$-edge-colourable. It was proved by Kronk, Radlowski and Franen \cite{Kronk19743} that every subcubic planar graph girth at least 8 is 3-edge-colorable. Only recently, Bonduelle and Kardo\v{s} \cite{BONDUELLE2022113002} improved this constant to 7. As mentioned before, admitting a decomposition into a linear forest and a matching is stronger than $3$-edge-colorability, and thus it is interesting to ask if one can strengthen the $3$-edge-colouring results to decompositions into linear forests and a matching. We show that for sufficiently large girth, this is the case:

\begin{Theorem}\label{girth9}
Every subcubic planar graph with girth at least $9$ decomposes into a linear forest and a matching. 
\end{Theorem}

We leave it as an open problem whether the girth bound can be improved. Rather surprisingly, it was proven by Montassier et al. that there are planar graphs of girth exactly $7$ that do not even decompose into a forest and a matching \cite{sndtck1d2}. Nevertheless, their construction requires vertices of degree $4$, and it does not appear easy to modify it to find subcubic graphs that do not decompose into a linear forest and a matching.

Our paper is structured as follows. First, in Section \ref{pos}, we prove Theorem \ref{21poly} relying on the fact that the so-called small gap general factor problem was shown to be polynomial-time solvable by Cornuéjols \cite{CORNU}. Next, we prove Theorem \ref{mainhard} in Section \ref{neg}. Finally, in Section \ref{gir9}, we prove Theorem \ref{girth9} using the discharging method.  

We define some notation that will be useful throughout. For a graph $G$, and a vertex $v \in V(G)$, we let $d_{G}(v)$ denote the degree of $v$ in $G$. If $S \subseteq E(G)$, we will  use $d_{S}(v)$ to mean the degree of $v$ in the graph induced by the edge set of $S$.

\section{Positive algorithmic results}\label{pos}

In this section, we prove Theorem \ref{21poly} by giving a polynomial-time reduction to the small gap general factor problem, and then relying on a theorem of Cornu{\'e}jols which shows this problem is solvable in polynomial time.

In order to explain the result of Cornuéjols we rely on, we first need some notation. A set $M \subseteq \mathbb{Z}_{\geq 0}$ is called a {\it small gap set} if for every $i \in \mathbb{Z}_{\geq 0}$ with $\{i,i+1\}\cap M=\emptyset$, we have either $\{0,\ldots,i+1\}\cap M=\emptyset$ or $\{i,i+1,\ldots\}\cap M=\emptyset$. The small gap general factor problem is formulated as follows:
\medskip

\noindent \textbf{Small gap general factor problem (SGGFC):}
\smallskip

\noindent\textbf{Input:} A graph $H$, a collection of small gap sets $\{M_v:v \in V(H)\}$.
\smallskip

\noindent\textbf{Question:} Is there a set $S \subseteq E(G)$ such that $d_S(v)\in M_v$ for all $v \in V(H)$?
\medskip

We rely on the following result of Cornuéjols \cite{CORNU}:

\begin{Theorem}
\label{Corn}
There exists a polynomial-time algorithm to solve SGGFC.
\end{Theorem}

\subsection{The construction}
\label{thereduction}
We are now ready to proceed to the main proof of Theorem \ref{21poly}. We start off with an easy but crucial observation:

\begin{obs}
If $G$ is an instance of $(2,1)$-BLFD and there exists a vertex of degree at least $4$ in $G$, then this is a no-instance of $(2,1)$-BLFD.
\end{obs}

\begin{proof}
In any $(2,1)$-BLFD decomposition, a vertex can be incident to at most one edge belonging to the matching, and to at most $2$ edges in the $2$-bounded linear forest. Thus if a graph has a vertex of degree at least four, there is no such $(2,1)$-BLFD.
\end{proof}

With that, for the rest of the section, we assume that we have an instance $G$ of $(2,1)$-BLFD and for every vertex $v \in V(G)$, we have $d_{G}(v) \leq 3$. To give the construction, we need some definitions.

Let $X_3=\{v \in V(G):d_G(v)=3\}$. Let $\mathcal{P}$ be the set containing the unique collection of paths in $G$ whose endvertices are of degree $1$ or $3$ in $G$ and all of whose interior vertices are of degree 2 in $G$ and also the set of cycles in $G$ containing at most one vertex in $X_3$. For a cycle $C$ in $\mathcal{P}$ containing a vertex of $X_{3}$, we will let the unique vertex $v \in V(C) \cap X_{3}$ be the \textit{endvertex} of $C$. For all cycles $C \in \mathcal{P}$ where $C \cap X_{3} = \emptyset$, we pick an arbitrary vertex $v\in V(C)$ and designate it as the endvertex of $C$. 
We refer to all vertices of a cycle $C$ that are not the endvertex as interior vertices.  Observe that $\{E(P):P \in \mathcal{P}\}$ is a partition of $E(G)$. For an illustration, see Figure \ref{corn1}.
\begin{figure}[h!]
    \centering
        \includegraphics[width=.9\textwidth]{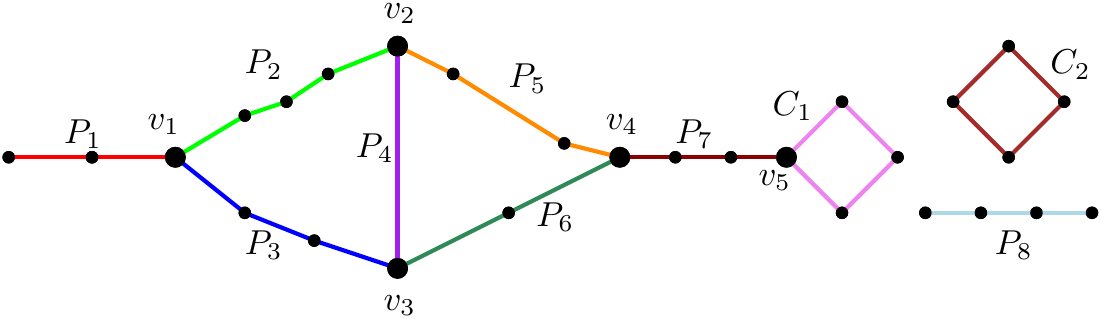}
        \caption{An example of the set $\mathcal{P}$ obtained from a subcubic graph. We have $\mathcal{P}=\{P_{1},\ldots,P_8,C_1,C_2\}$ and the colors of the edges indicate which element of $\mathcal{P}$ the edge belongs to. The vertices in $X_3$ are drawn bigger and marked as $v_1,\ldots,v_5$.}\label{corn1}
\end{figure}

    \begin{figure}[h!]
    \centering
        \includegraphics[width=.4\textwidth]{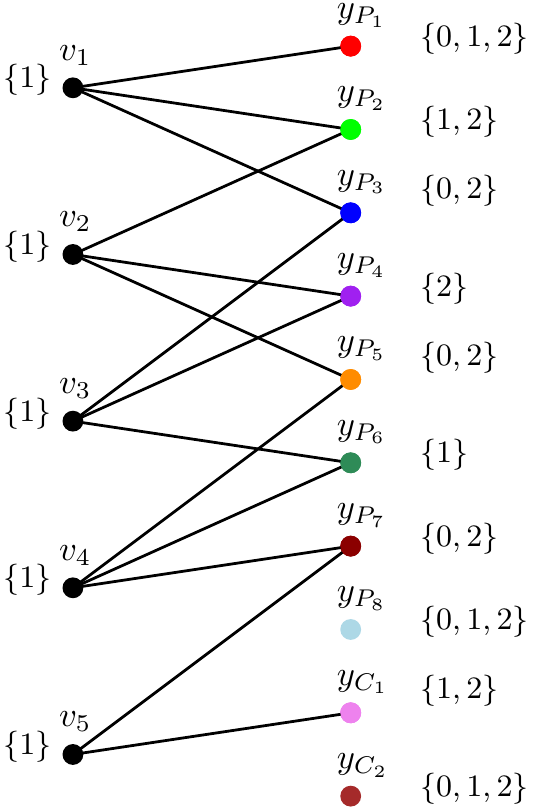}
        \caption{An example of the graph $H$ obtained from the graph $G$ depicted in Figure \ref{corn1}. The sets $M_v$ are indicated next to the corresponding vertices.}\label{corn2}
\end{figure}

   We can now define the instance of SGGFC associated to $G$.
    \begin{definition}
    Given the instance $G$ of $(2,1)$-BLFD, we define an instance $(H,\{M_v:v \in V(H)\})$ of SGGFC in the following manner. The graph $H$ is a bipartite graph with bipartition $(A,B)$ such that $A = X_{3}$ and $B = \{y_{P} \, | \, P \in \mathcal{P}\}$. We define $E(H)$ so that:
    \begin{itemize}
    \item{For every $P \in \mathcal{P}$ and every endvertex $u$ of $P$ that is in $X_3$, we let $E(H)$ contain an edge linking $u$ and $y_P$.}
    \end{itemize}
    We define $\{M_v:v \in V(H)\}$ in the following manner:
    \begin{itemize}
    \item{For every $v \in X_3$, we set $M_v=\{1\}$.}
    \item{For every $P \in \mathcal{P}$ of length 1 all of whose endvertices are in $X_3$, we set $M_{y_P}=\{2\}$.}
    \item{For every  $P \in \mathcal{P}$ of length 2 all of whose endvertices are in $X_3$, we set $M_{y_P}=\{1\}$.}
    \item{For every $P \in \mathcal{P}$ of length 3 all of whose endvertices are in $X_3$, we set $M_{y_P}=\{0,2\}$.}
    \item{For every $P \in \mathcal{P}$ of length 4 all of whose endvertices are in $X_3$, we set $M_{y_P}=\{1,2\}$.}
    \item{For all remaining $P \in \mathcal{P}$, we set $M_{y_P}=\{0,1,2\}$.}
    \end{itemize}
    \end{definition}
     
    For an illustration of the definition of $(H,\{M_v:v \in V(H)\})$, see Figure \ref{corn2}.

\subsection{Proof of reduction}
 In this subsection, we  show that $(H,\{M_v:v \in V(H)\})$ is a yes instance of SGGFC if and only if $G$ is a yes instance of $(2,1)$-BLFD.

\begin{Claim}
If $(H,\{M_v:v \in V(H)\})$ is a yes instance of SGGFC, then $G$ is a yes instance of $(2,1)$-BLFD.
\end{Claim}
 \begin{proof} 
By the assumption, there is a set $S \subseteq E(H)$ with $d_S(v)\in M_v$ for all $v \in V(H)$. We now create a $(2,1)$-bounded linear forest decomposition $(F_2,F_1)$ of $G$, where $F_{1}$ is the matching, and $F_{2}$ is a linear forest where each component has at most two edges. As $\{E(P):P \in \mathcal{P}\}$ is a partition of $E(G)$, it suffices to describe how to split $E(P)$ into edges of $E(F_1)$ and $E(F_2)$ for all $P \in \mathcal{P}$. We do this now.

    Let $P\in \mathcal{P}$ with endvertices $u$ and $v$ and let $u=z_0,z_1,\ldots,v=z_t$ be the vertices in $V(P)$ in the order they appear in $P$, where $u=v$ if $P$ is a cycle. We split into cases based on the endvertices of $P$. \\
\textbf{Case 1: No endvertex of $P$ is contained in $X_{3}$:}\\
     We add the set $\{z_iz_{i+1}: 0\leq i \leq t-1,i \text{ odd}\}$ to $E(F_{1})$ and we add $E(P) - E(F_{1})$ to $E(F_{2})$. For an illustration, see Figure \ref{corn3}$(a)$. \\
    \textbf{Case 2: $P$ contains two endvertices $u,v$, where $d_{G}(u) =3$ and $d_{G}(v) = 1$}\\
     If $uy_P \in S$, we add the set $\{z_iz_{i+1}: 0\leq i \leq t-1,i \text{ even}\}$ to $E(F_{1})$ and we add $E(P) -E(F_{1})$ to $E(F_{2})$. For an illustration, see Figure \ref{corn3}$(b)$. \\
     If $uy_P \notin S$, we let $E(F_1)$ contain $\{z_iz_{i+1}: 0\leq i \leq t-1,i \text{ odd}\}$ and we let $E(F_2)$ contain $E(P)-E(F_1)$. For an illustration, see Figure \ref{corn3}$(c)$. \\
\textbf{Case 3: The endvertices of $P$ are both in $X_{3}$, and $uv \in E(G)$.} \\
 In this case, we add $uv$ to $E(F_{1})$. For an illustration, see Figure \ref{corn3}$(d)$.\\
 \textbf{Case 4: The endvertices of $P$ are both in $X_{3}$, and $t =2$.} \\
In this case, it follows that exactly one of $uy_P$ and $vy_P$ is contained in $S$, and without loss of generality suppose $uy_{P}$ is. We add $uz_{1}$ to $E(F_1)$ we add $z_{1}v$ to $E(F_2)$. For an illustration, see Figure \ref{corn3}$(e)$. \\
\textbf{Case 5: The endvertices of $P$ are both in $X_{3}$, and $t=3$}\\
In this case, $M_{y_P}=\{0,2\}$, and we obtain that either both or none of $uy_P$ and $vy_P$ are contained in $S$. We treat these cases separately. 

If both $uy_{P}$ and $vy_{P}$ are contained in $S$, observe that $P$ is a path as otherwise $u =v$. We add $uz_{1}$ and $vz_{2}$ to $E(F_{1})$ and we add $z_{1}z_{2}$ to $E(F_{2})$. For an illustration, see Figure \ref{corn3}$(f)$.

If none of $uy_P$ and $vy_P$ are contained in $S$, we add $z_{1}z_{2}$ to $E(F_{1})$ and we add $uz_{1}$ and $vz_{2}$ to $E(F_{2})$ . For an illustration, see Figure \ref{corn3}$(g)$.\\
\textbf{Case 6: The endvertices of $P$ are both in $X_{3}$ and $t=4$}\\
In this case, we have that $M_{y_P}=\{1,2\}$ and we obtain that at least one of $uy_P$ and $vy_P$ is contained in $S$. 

If both $uy_P$ and $vy_P$ are contained in $S$, observe that $P$ is a path. We add $uz_{1}$ and $z_{3}v$ to $E(F_1)$ and we add $z_{1}z_{2}$ and $z_{2}z_{3}$ to $E(F_2)$. For an illustration, see Figure \ref{corn3}$(h)$.

If exactly one of $uy_P$ and $vy_P$, say $uy_P$, is contained in $S$, we  add $uz_{1}$ and $z_{2}z_{3}$ to $E(F_1)$ and we add $z_{1}z_{2}$ and $z_{3}v$ to $E(F_2)$. For an illustration, see Figure \ref{corn3}$(i)$.\\
\textbf{Case 7: The endvertices of $P$ are both in $X_{3}$ and $t\geq 5$.} \\
There are three subcases to consider, depending on if the edges $uy_{P}$ and $vy_{P}$ are in $S$.

First suppose that $uy_P,vy_P \in S$, and as before, observe that $P$ is a path. If $t$ is odd, we add the set $\{z_iz_{i+1}: 0\leq i \leq t-1,i \text{ even }\}$  to $E(F_{1})$, and we add $E(P)-E(F_{1})$ to $E(F_{2})$. For an illustration, see Figure \ref{corn3}$(j)$. If $t$ is even, we add  $\{z_iz_{i+1}: 3\leq i \leq t-1,i \text{ odd }\}\cup uz_1$ to $E(F_{1})$ and we add $E(P) - E(F_{2})$ to $E(F_{2})$. For an illustration, see Figure \ref{corn3}$(k)$.

Second, suppose that exactly one of $uy_P$ and $vy_P$, is contained in $S$, without loss of generality let it be $uy_P$. If $t$ is odd, we add  $\{z_iz_{i+1}: 3\leq i \leq t-1,i \text{ odd }\}\cup vz_1$ to $E(F_1)$ and we add $E(P)-E(F_{1})$ to $E(F_2)$. For an illustration, see Figure \ref{corn3}$(\ell)$. If $t$ is even, we add $\{z_iz_{i+1}: 0\leq i \leq t-1,i \text{ even }\}$ to the set $E(F_1)$ and we add $E(P)-E(F_{1})$ to  $E(F_2)$. For an illustration, see Figure \ref{corn3}$(m)$.

\begin{figure}[h!]
    \centering
        \includegraphics[width=.9\textwidth]{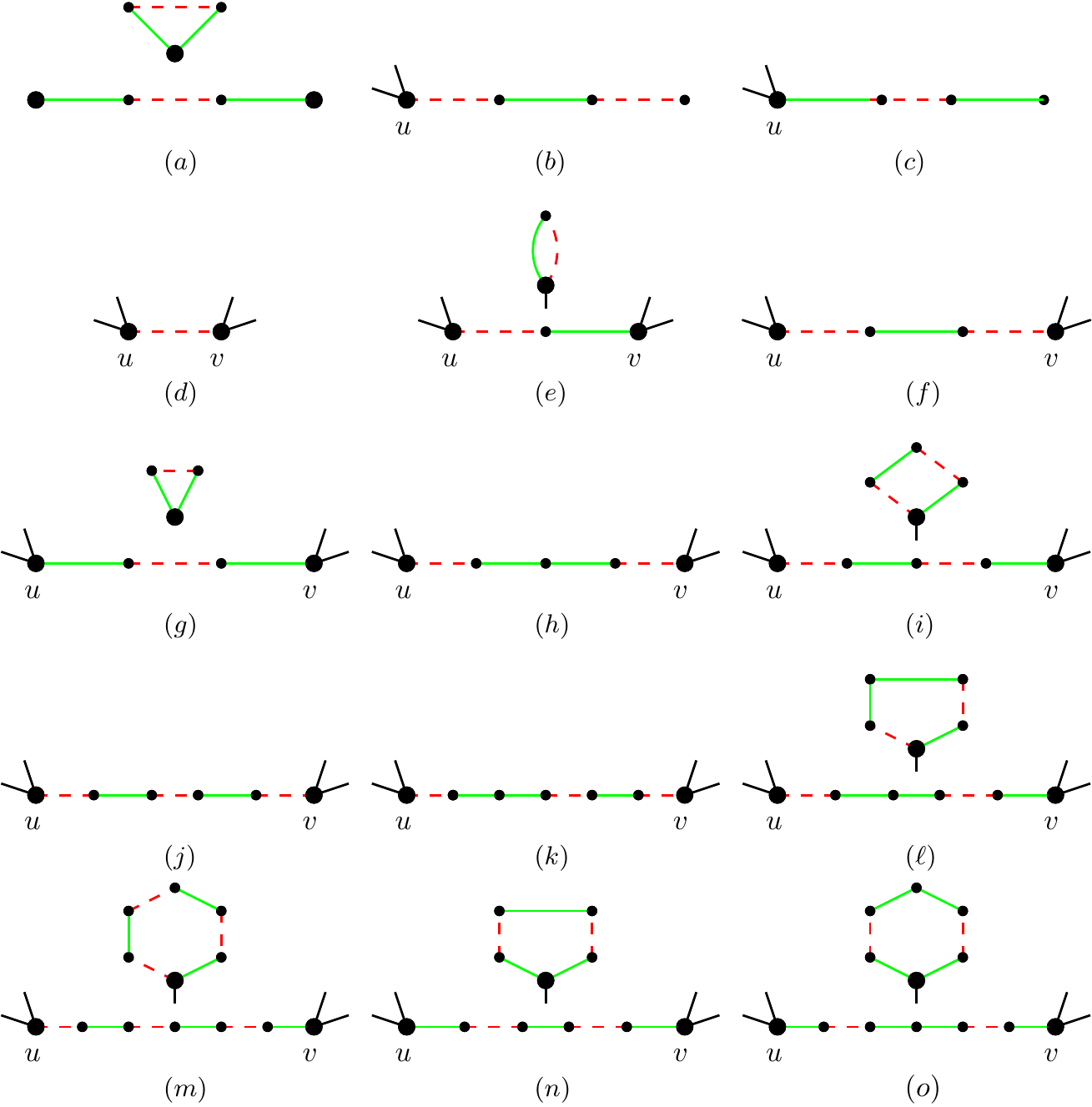}
        \caption{An illustration of how the forests $F_1$ and $F_2$ are created in different cases. The dashed red edges are in $E(F_1)$ and the solid green edges are in $E(F_2)$.}\label{corn3}
\end{figure}
Finally, suppose that none of $uy_P$ and $vy_P$ are contained in $S$. If $t$ is odd, we add $\{z_iz_{i+1}: 0\leq i \leq t-1,i \text{ odd}\}$ to $E(F_1)$ and we add $E(P)-E(F_{1})$ to $E(F_2)$. For an illustration, see Figure \ref{corn3}$(n)$. If $t$ is even, we add $\{z_iz_{i+1}: 4\leq i \leq t-1,i \text{ even}\}\cup z_1z_2$ to $E(F_1)$ and we add $E(P)-E(F_1)$ to $E(F_{2})$. For an illustration, see Figure \ref{corn3}$(o)$. \\

With this, we have a description of $(F_2,F_1)$. Observe that for every path $P \in \mathcal{P}$ and every endvertex $u$ of $P$, we have that the edge of $P$ incident to $u$ is contained in $E(F_1)$ if and only if $uy_P \in S$. Also, for every cycle $C \in \mathcal{P}$ with endvertex $u$, we have that the number of edges $C$ incident to $u$ is exactly the number of edges linking $u$ and $y_P$ that are in $S$.
As $M_u=\{1\}$ for all $u \in X_3$, this yields that every $u \in X_3$ is incident to exactly one edge in $F_1$. It further follows by construction that every interior vertex of $P$ is incident to at most one edge in $E(F_1)$ for all $P \in \mathcal{P}$. Hence $F_1$ is a matching. Now consider a connected component $Q$ of $F_2$. By construction, if $E(Q)\subseteq E(P)$ for some $P \in \mathcal{P}$, then $Q$ is a path of length at most 2. Otherwise, there is some $v \in X_3$ and $P_1,P_2 \in \mathcal{P}$ such that the edge of $E(P_i)$ incident to $v$ is in $E(Q)$. As $d_{F_1(v)}=1$, we have $d_{F_2(v)}=2$. Further, it follows by construction that the length of $P_i$ is at least 2 and that $E(P_i)\cap E(Q)$ contains only a single edge. Hence $F_2$ is a 2-bounded linear forest, so $(F_2,F_1)$ is a $(2,1)$ bounded linear forest decomposition of $G$. Hence $G$ is a yes instance of $(2,1)$-BLFD.
\end{proof}

\begin{Claim}
If $G$ is a yes instance of $(2,1)$-BLFD, then $(H,\{M_v:v \in V(H)\})$ is a yes-instance of SGGFC.
\end{Claim}
\begin{proof}
Let $(F_{2},F_{1})$ be a $(2,1)$-bounded linear forest decomposition of $G$. We now define a set $S \subseteq E(H)$. For every path $P \in \mathcal{P}$ and every endvertex $u$ of $P$, we let $S$ contain the edge $uy_P$ if the edge in $E(P)$ that is incident to $u$ is contained in $E(F_1)$. For every cycle $C \in \mathcal{P}$ whose endvertex is $u$, we let $S$ contain exactly as many of the edges linking $u$ and $y_P$ as $E(F_1)$ contains edges of $E(C)$ incident to $u$.

It remains to show that $d_S(u)\in M_u$ for all $u \in V(H)$. First suppose that $u \in X_{3}$. As $(F_2,F_1)$ is a $(2,1)$-bounded linear forest decomposition of $G$, we obtain that $u$ is incident to exactly one edge of $E(F_1)$. This yields $d_S(u)=1$.

Now let $P\in \mathcal{P}$ and let $\{u=z_0,z_1,\ldots,v=z_t\}$ be the vertices in $V(P)$ in the order they appear in $P$ where $u=v$ if $P$ is a cycle. We split into cases.


 \textbf{Case 1: The length of $P$ is $1$ and all endvertices of $P$ are contained in $X_{3}$}\\
  As $(F_2,F_1)$ is a $(2,1)$-bounded linear forest decomposition of $G$, we obtain $d_{F_2}(u)=d_{F_2}(v)=2$, so both $u$ and $v$ are incident to an edge of $E(F_2)-E(P)$. Hence, if $uv \in E(F_2)$, then $F_2$ contains a path of length 3 or a cycle, a contradiction. We hence obtain $uv \in E(F_1)$, so both $uy_P$ and $vy_P$ are contained in $S$. This yields $d_S(y_P)=2\in M_{y_P}$.

\textbf{Case 2: The length of $P$ is 2 and all endvertices of $P$ are contained in $X_3$.}\\
As $F_1$ is a matching at least one of $uz_1$ and $z_1v$ is contained in $E(F_2)$. As $d_{F_2}(u)=2$, we have that $u$ is incident to an edge of $E(F_2)-E(P)$. Hence, if $uz_1,z_1v \in E(F_2)$, then $F_2$ contains a path of length 3 or a cycle, a contradiction. We hence obtain that exactly one of $uz_1$ and $z_1v$ is contained in $E(F_1)$. This yields $d_S(y_P)=1\in M_{y_P}$.

\textbf{Case 3: The length of $P$ is 3 and all endvertices of $P$ are contained in $X_3$. }\\ Suppose for the sake of a contradiction that exactly one of $uz_1$ and $z_2v$, say $uz_1$, is contained in $E(F_1)$. As $F_1$ is a matching, we obtain that $z_1z_2$ is contained in $E(F_2)$. As $d_{F_2}(v)=2$, we have that $v$ is incident to an edge of $E(F_2)-E(P)$. Hence $F_2$ contains a path of length 3, a contradiction. We hence obtain that either both or none of $uz_1$ and $z_2v$ are contained in $E(F_1)$. This yields $d_S(y_P)\in \{0,2\}= M_{y_P}$.

\textbf{Case 4: The length of $P$ is 4 and all endvertices of $P$ are contained in $X_3$.}\\ Suppose for the sake of a contradiction that both of $uz_1$ and $z_2v$ are contained in $E(F_2)$. As $d_{F_2}(u)=2$, we have that $u$ is incident to an edge of $E(F_2)-E(P)$. Hence, as $F_2$ does not contain a path of length 3, we obtain $z_1z_2 \in E(F_1)$. We similarly obtain $z_2z_3 \in E(F_1)$. This contradicts $F_1$ being a matching. We hence obtain that at least one of $uz_1$ and $z_2v$ is contained in $E(F_1)$. This yields $d_S(y_P)\in \{1,2\}= M_{y_P}$.

\textbf{Case 5: All other cases} \\
For all $P \in \mathcal{P}$ which are of none of the forms considered above, we trivially have $d_S(y_P)\in M_{y_P}$.\\

We hence have $d_S(v)\in M_{v}$ for all $v \in V(H)$, so $(H,\{M_v:v \in V(H)\})$ is a yes-instance of SGGFC.
\end{proof}
Theorem \ref{21poly} now follows by applying Theorem \ref{Corn} to $(H,\{M_v:v \in V(H)\})$, which can be constructed in polynomial time given $G$.

\section{Hardness results}\label{neg}
This section is concerned with proving Theorem \ref{mainhard}. In Section \ref{prelhard}, we give the problems we reduce from. In Section \ref{simgad}, we give some simple gadgets that will prove useful in the main reductions later. After, in Sections \ref{k1hard}, \ref{kinfhard} and \ref{klhard}, we prove the NP-completeness of $(k,\ell)$-BLFD for several sets of values for $k$ and $\ell$. Together with Theorem \ref{inftyinfty}, these results yield Theorem \ref{mainhard}. 

\subsection{Preliminaries}\label{prelhard}

We here describe the two variants of the satisfiablility problem we need for our reductions.
\medskip

\noindent \textbf{$(3,B2)$-SAT :}
\smallskip

\noindent\textbf{Input:} A set of variables $X$, a set of clauses $\mathcal{C}$ such that $|C|=3$ for all $C \in \mathcal{C}$ and for every $x \in X$, both $x$ and $\bar{x}$ are contained in exactly two clauses.
\smallskip

\noindent\textbf{Question:} Is there a truth assignment $\phi: X \rightarrow \{TRUE,FALSE\}$ such that every clause contains at least one true literal?
\medskip

\noindent \textbf{Monotonous not all equal 3SAT (MNAE3SAT):}
\smallskip

\noindent\textbf{Input:} A set of variables $X$, a set of clauses $\mathcal{C}$ such that every $C \in \mathcal{C}$ contains exactly 3 positive literals.
\smallskip

\noindent\textbf{Question:} Is there a truth assignment $\phi: X \rightarrow \{TRUE,FALSE\}$ such that every clause contains at least one true and at least one false literal?
\medskip

For both these problems, we call an assignment $\phi$ with the desired properties {\it satisfying}. Both of these problems are known to be NP-complete. 

\begin{Theorem}[\cite{Berman2003ApproximationHO}]
\label{scott}
$(3,B2)$-SAT is NP-complete.
\end{Theorem}

\begin{Theorem}[\cite{MNAE3SAT}]
\label{mono}
MNAE3SAT is NP-complete.
\end{Theorem}

\subsection{Simple gadgets}\label{simgad}

In this section, we describe a collection of simple gadgets we need for the main reductions in Sections \ref{k1hard}, \ref{kinfhard} and \ref{klhard}.

A {\it long $1$-forcer} consists of a path of length 4 the first and last edge of which are doubled and a vertex called the \textit{tip vertex} of the long $1$-forcer which is linked to the third vertex of the initial path by an edge.
An illustration can be found in Figure \ref{image_0}. 
\begin{figure}[h!]
    \centering
        \includegraphics[width=.5\textwidth]{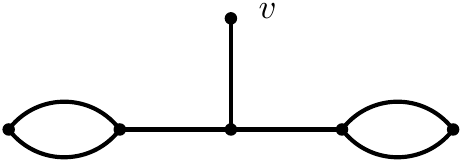}
        \caption{A long $1$-forcer. The vertex marked $v$ is the tip vertex of the short long $1$-forcer.}\label{image_0}
\end{figure}The decisive property of a long 1-forcer is the following:

\begin{Proposition}\label{long1}
Let $G$ be a long $1$-forcer and $k \in \{4,\ldots\}\cup \infty$. Then $G$ has a unique $(k,1)$-bounded linear forest decomposition $(F_k,F_1)$ in which the tip vertex $v$ of $G$ satisfies $d_{F_k}(v)=0$ and $d_{F_1}(v)=1$.
\end{Proposition}
\begin{proof}
Let $(F_k,F_1)$ be a $(k,1)$-bounded linear forest decomposition of $G$. Clearly, for each pair of parallel edges, one of them is contained in $E(F_k)$ and one of them is contained in $E(F_1)$. Hence, as $F_1$ is a matching, we obtain that the two middle edges of the initial path of $G$ are contained in $E(F_k)$. As $F_k$ is a linear forest, we obtain that $(F_k,F_1)$ has the desired properties. Further, it is easy to see that $(F_k,F_1)$ is indeed a $(k,1)$-bounded linear forest decomposition of $G$. 
\end{proof}
For an illustration, see Figure \ref{image_-1}.
\begin{figure}[h!]
    \centering
        \includegraphics[width=.5\textwidth]{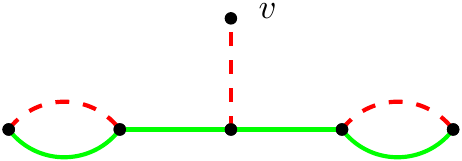}
        \caption{An illustration of the unique $(k,1)$-bounded linear forest decomposition $(F_k,F_1)$ of a long 1-forcer for any $k \in \{4,\ldots\}\cup \infty$. The dashed red edges are in $E(F_{1})$ and the solid green edges are in $E(F_{k})$. The vertex marked $v$ is the tip vertex of the long $1$-forcer.}\label{image_-1}
\end{figure}

A {\it short $1$-forcer} is obtained from a vertex, called the tip vertex of the short $1$-forcer, and linking it to the tip vertex of a long 1-forcer. See Figure \ref{image_-2} for an illustration. 
\begin{figure}[h!]
    \centering
        \includegraphics[width=.5\textwidth]{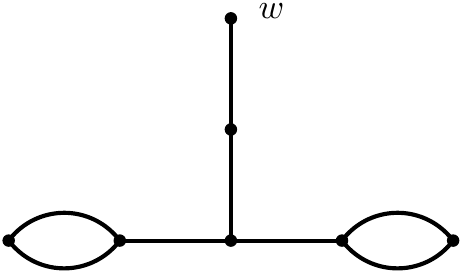}
        \caption{A short $1$-forcer. The vertex marked $w$ is the tip vertex of the short $1$-forcer.}\label{image_-2}
\end{figure} Using Proposition \ref{long1}, we get the decisive property of short $1$-forcers:

\begin{Proposition}\label{short1}
Let $G$ be a short $1$-forcer and $k \in \{4,\ldots\}\cup \infty$. Then $G$ has a unique $(k,1)$-bounded linear forest decomposition $(F_k,F_1)$ in which the tip vertex $v$ of $G$ satisfies $d_{F_1}(v)=0$ and is contained in a unique maximal path of $F_k$ which is of length 1.
\end{Proposition}

For an illustration, see Figure \ref{image_-3}.
\begin{figure}[h!]
    \centering
        \includegraphics[width=.5\textwidth]{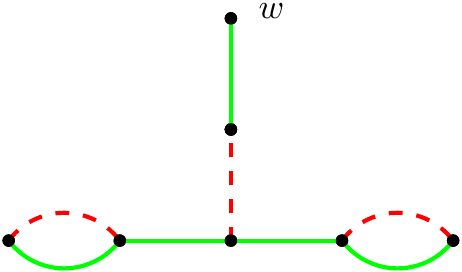}
        \caption{An illustration of the unique $(k,1)$-bounded linear forest decomposition $(F_k,F_1)$ of a short 1-forcer for any $k \in \{4,\ldots\}\cup \infty$. The dashed red edges are in $E(F_{1})$ and the solid green edges are in $E(F_{k})$. The vertex marked $w$ is the tip vertex of the short $1$-forcer.}\label{image_-3}
\end{figure}
\medskip

Let $k,\ell$ be positive integers with $k > \ell \geq 2$. A {\it short $(k,\ell)$-forcer} is obtained from a path $v_0\ldots v_k$ by doubling the edge $v_iv_{i+1}$ for all $i=0,\ldots,k-1$ that do not satisfy $i=k-1-\mu(\ell+1)$ for some integer $\mu \in \{0,\ldots,\lfloor\frac{k-1}{\ell+1}\rfloor\}$. We call $v_k$ the tip vertex of the short $(k,\ell)$-forcer. See Figure \ref{image_1} for an illustration. 
\begin{figure}[h!]
    \centering
        \includegraphics[width=.03\textwidth]{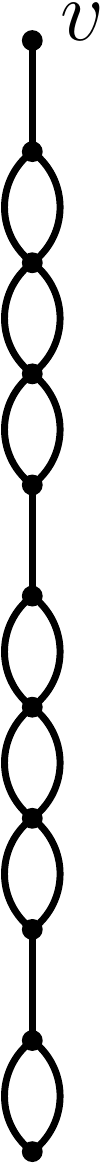}
        \caption{A short $(10,3)$-forcer. The vertex marked $v$ is the tip vertex of the short $(10,3)$-forcer.}\label{image_1}
\end{figure}The decisive property of short $(k,\ell)$-forcers is the following:

\begin{Proposition}\label{short}
Let $G$ be a short $(k,\ell)$-forcer for some positive integers with $k >\ell \geq 2$. Then $G$ has a unique $(k,\ell)$-bounded linear forest decomposition $(F_k,F_\ell)$ in which the tip vertex $v$ of $G$ satisfies $d_{F_\ell}(v)=0$ and is an endpoint of a path of length $k$ in $F_k$.
\end{Proposition}
\begin{proof}
Let $(F_k,F_\ell)$ be a $(k,\ell)$-bounded linear forest decomposition of $G$. Clearly, for each pair of parallel edges, one of them is contained in $E(F_k)$ and one of them is contained in $E(F_\ell)$. Hence, as $F_\ell$ is an $\ell$-bounded linear forest, we obtain that $v_{i}v_{i+1}\in E(F_k)$ for all $i$ that satisfy $i=k-1-\mu(\ell+1)$ for some integer $\mu \in \{0,\ldots,\lceil\frac{k-1}{\ell+1}\rceil\}$. Hence $(F_k,F_\ell)$ has the desired properties. Further, it is easy to see that $(F_k,F_\ell)$ is indeed a $(k,\ell)$-bounded linear forest decomposition of $G$. 
\end{proof}

\begin{figure}[h!]
    \centering
        \includegraphics[width=.03\textwidth]{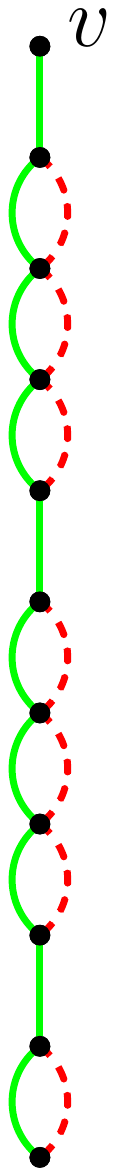}
        \caption{An illustration of the unique $(10,3)$-bounded linear forest decomposition of a short $(10,3)$-forcer. The dashed red edges are in $E(F_{3})$ and the solid green edges are in $E(F_{10})$. The vertex marked $v$ is the tip vertex of the short $(10,3)$-forcer.}\label{image_2}
\end{figure}

For an illustration, see Figure \ref{image_2}. We use short $(k,\ell)$-forcers to obtain a further gadget with a similar role. Let $G_1,G_2$ be two short $(k,\ell)$-forcers whose tip vertices are $v_1$ and $v_2$, respectively. We now obtain a  {\it long $(k,\ell)$-forcer} by adding a new vertex $w$ and the edges $v_1w$ and $v_2w$. We call $w$ the tip vertex of the long $(k,\ell)$-forcer. For an illustration, see Figure \ref{image_3}.
\begin{figure}[h!]
    \centering
        \includegraphics[width=.4\textwidth]{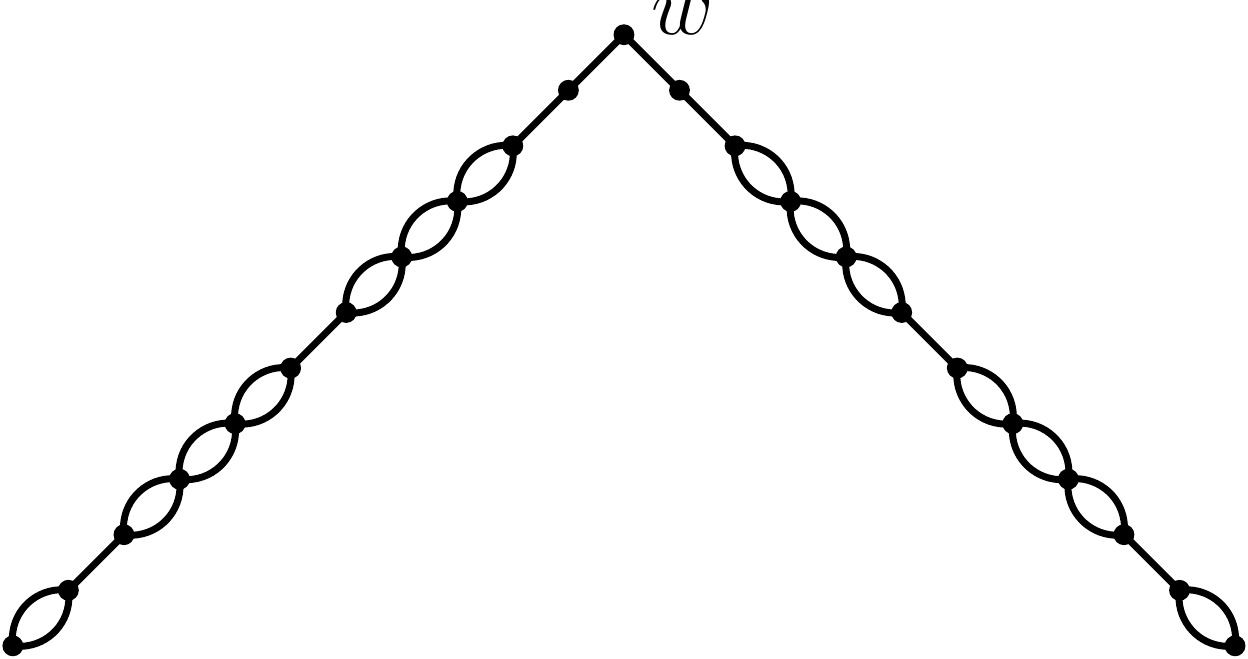}
        \caption{A long $(10,3)$-forcer. The vertex marked $w$ is the tip vertex of the long $(10,3)$-forcer.}\label{image_3}
\end{figure}
The decisive property of long $(k,\ell)$-forcers is the following:

\begin{Proposition}\label{long}
Let $G$ be a long $(k,\ell)$-forcer for some positive integers with $k >\ell \geq 2$. Then $G$ has a unique $(k,\ell)$-bounded linear forest decomposition $(F_k,F_\ell)$ in which the tip vertex $w$ of $G$ satisfies $d_{F_k}(w)=0$ and $d_{F_\ell}(w)=2$.
\end{Proposition}
\begin{proof}
Let $(F_k,F_\ell)$ be a $(k,\ell)$-bounded linear forest decomposition of $G$. By Proposition \ref{short}, for $i=1,2$, the restriction of $(F_k,F_\ell)$ to $G_i$ is exactly as described in Proposition \ref{short}, so $v_i$ is the last vertex of a path of length $k$ in $F_k$ that is entirely contained in $G_i$. As $F_k$ is a $k$-bounded linear forest, we obtain that $\{v_1w,v_2w\}\subseteq E(F_\ell)$. Hence the statement follows. Further, it is easy to see that $(F_k,F_\ell)$ is indeed a $(k,\ell)$-bounded linear forest decomposition of $G$. 
\end{proof}
For an illustration, see Figure \ref{image_4}.
\begin{figure}[h!]
    \centering
        \includegraphics[width=.4\textwidth]{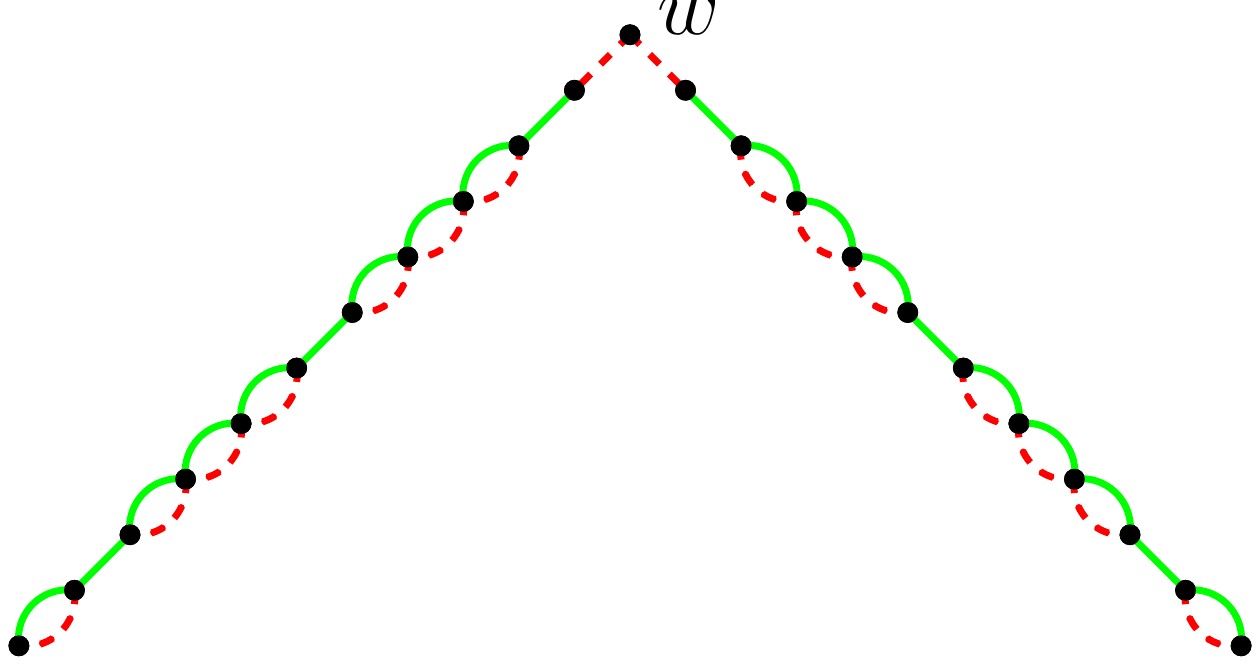}
        \caption{An illustration of the unique $(10,3)$-bounded linear forest decomposition of a long $(10,3)$-forcer. The dashed red edges are in $E(F_{3})$ and the solid green edges are in $E(F_{10})$. The vertex marked $w$ is the tip vertex of the long $(10,3)$-forcer.}\label{image_4}
\end{figure}

The next simple gadget will be used for the case when both linear forests have the same length bound. For some $k\geq 2$, a {\it symmetric $k$-forcer} is a path of length $k$ in which all edges except the last one are doubled. The unique vertex of degree 1 is called the tip vertex of the symmetric $k$-forcer. For an illustration, see Figure \ref{image_5}. 
\begin{figure}[h!]
    \centering
        \includegraphics[width=.08\textwidth]{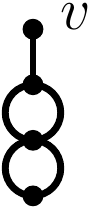}
        \caption{A symmetric $3$-forcer. The vertex marked $v$ is the tip vertex of the symmetric $3$-forcer.}\label{image_5}
\end{figure}
The decisive property of symmetric $k$-forcers, which is easy to see,  is the following: 

\begin{Proposition}\label{symmetric}
Let $G$ be a symmetric $k$-forcer for some positive integer $k \geq 2$. Then $G$ has a $(k,k)$-bounded linear forest decomposition $(F,F')$ which is unique up to exchanging $F$ and $F'$ and in which the tip vertex $v$ of $G$ satisfies $d_{F}(v)=0$ and is the last vertex of a path of length $k$ in $F'$.
\end{Proposition}
For an illustration, see Figure \ref{image_6}.
\begin{figure}[h!]
    \centering
        \includegraphics[width=.08\textwidth]{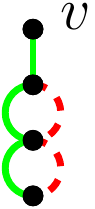}
        \caption{An illustration of a $(3,3)$-bounded linear forest decomposition of a symmetric $3$-forcer. The dashed red edges are in $E(F)$ and the solid green edges are in $E(F')$. The vertex marked $v$ is the tip vertex of the symmetric $3$-forcer.}\label{image_6}
\end{figure}

We now construct two more gadgets that we need for the case that one of the linear forests is unrestricted. For some integer $k \geq 2$, let $G_1,\ldots,G_4$ be four short $(k+1,k)$-forcers whose tip vertices are $v_1,\ldots,v_4$, respectively. We now obtain a  {\it long $(\infty,k)$-forcer} by first identifying $v_1$ and $v_2$ into a new vertex $w_1$ and identifying $v_3$ and $v_4$ into a new vertex $w_2$ and then adding a new vertex $w$ and the edges $w_1w$ and $w_2w$. We call $w$ the tip vertex of the long $(\infty,k)$-forcer. For an illustration, see Figure \ref{inf_1}.

\begin{figure}[h!]
    \centering
        \includegraphics[width=.2\textwidth]{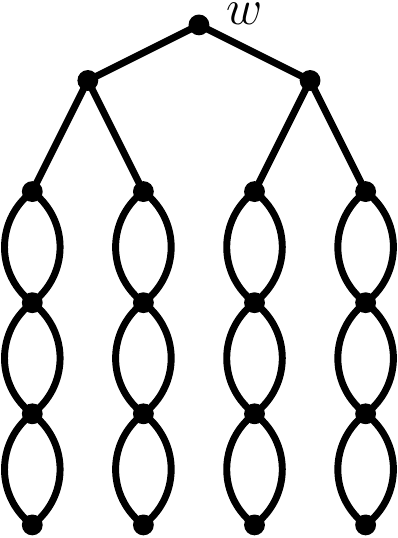}
        \caption{A long $(\infty,3)$-forcer. The vertex marked $w$ is the tip vertex of the long $(\infty,3)$-forcer.}\label{inf_1}
\end{figure}

The decisive property of long $(\infty,k)$-forcers is the following:

\begin{Proposition}\label{longinf}
Let $G$ be a long $(\infty,k)$-forcer for some positive integer $k \geq 2$. Then $G$ has a unique $(\infty,k)$-bounded linear forest decomposition $(F_\infty,F_k)$ in which the tip vertex $w$ of $G$ satisfies $d_{F_\infty}(w)=0$ and $d_{F_k}(w)=2$.
\end{Proposition}
\begin{proof}
Let $(F_\infty,F_k)$ be a $(\infty,k)$-bounded linear forest decomposition of $G$. By a similar argument to the one in the proof of Proposition \ref{short}, for $i=1,2$, the restriction of $(F_\infty,F_k)$ to $G_i$ is exactly as described in Proposition \ref{short}, so $v_i$ is contained in a path in $F_\infty$ that is entirely contained in $G_i$. This yields $d_{F_\infty}(w_1)=d_{F_\infty}(w_2)=2$. As $F_\infty$ is a linear forest, we obtain that $\{w_1w,w_2w\}\subseteq E(F_k)$. Hence the statement follows. Further, it is easy to see that $(F_\infty,F_k)$ is indeed a $(\infty,k)$-bounded linear forest decomposition of $G$. 
\end{proof}

For an illustration, see Figure \ref{inf_2}.
\begin{figure}[h!]
    \centering
        \includegraphics[width=.2\textwidth]{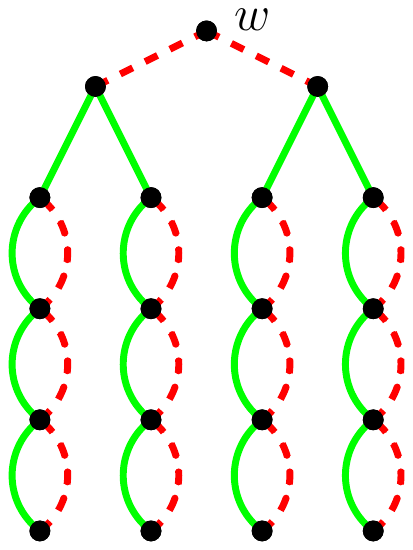}
        \caption{An illustration of a $(\infty,3)$-bounded linear forest decomposition $(F_\infty,F_3)$ of a long $(\infty,3)$-forcer. The dashed red edges are in $E(F_3)$ and the solid green edges are in $E(F_\infty)$. The vertex marked $w$ is the tip vertex of the long $(\infty,3)$-forcer.}\label{inf_2}
\end{figure}

For some positive integer $k\geq 2$, an {\it $(\infty,k)$-path forcer} is a path $P$ of length $k-1$ all of whose vertices except the last one have been identified with the tip vertices of two short $(k+1,k)$-forcers. The unique vertex of degree 1 of an $(\infty,k)$-path forcer is called the tip vertex of the $(\infty,k)$-path forcer.
 For an illustration, see Figure \ref{inf_3}.

\begin{figure}[h!]
    \centering
        \includegraphics[width=.4\textwidth]{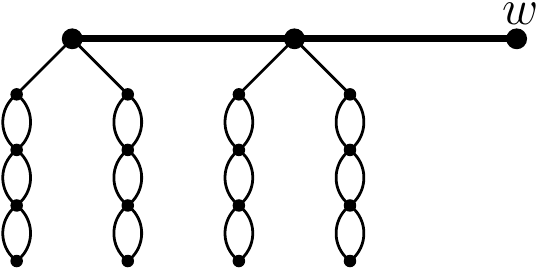}
        \caption{An $(\infty,3)$-path forcer. The vertex marked $w$ is the tip vertex of the $(\infty,3)$-path forcer.}\label{inf_3}
\end{figure}

The decisive property of $(\infty,k)$-path forcers is the following:

\begin{Proposition}\label{path}
Let $G$ be an $(\infty,k)$-path forcer for some positive integer $k \geq 2$. Then $G$ has a unique $(\infty,k)$-bounded linear forest decomposition $(F_\infty,F_k)$ in which the tip vertex $w$ of $G$ satisfies $d_{F_\infty}(w)=0$, is the last vertex of a path of length $k-1$ of $F_k$ and is contained in a path of length $k-1$ of $F_k$. 
\end{Proposition}
\begin{proof}
Let $(F_\infty,F_k)$ be a $(\infty,k)$-bounded linear forest decomposition of $G$. By a similar argument to the one in the proof of Proposition \ref{short}, for $i=1,2$, the restriction of $(F_\infty,F_k)$ to the short $(k+1,k)$-forcers attached to the vertices in $V(P)-w$  is exactly as described in Proposition \ref{short}, so $v$ is contained two paths in $F_\infty$ which are edge-disjoint from $P$. This yields $E(P)\subseteq E(F_k)$. Hence the statement follows. Further, it is easy to see that $(F_\infty,F_k)$ is indeed a $(\infty,k)$-bounded linear forest decomposition of $G$. 
\end{proof}

For an illustration, see Figure \ref{inf_4}.
\begin{figure}[h!]
    \centering
        \includegraphics[width=.4\textwidth]{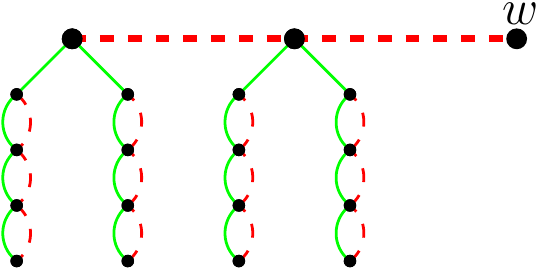}
        \caption{An illustration of a $(\infty,3)$-bounded linear forest decomposition $(F_\infty,F_3)$ of a $(\infty,3)$-path forcer. The dashed red edges are in $E(F_3)$ and the solid green edges are in $E(F_\infty)$. The vertex marked $w$ is the tip vertex of the $(\infty,3)$-path forcer.}\label{inf_4}
\end{figure}

\subsection{Cases involving a matching} \label{k1hard}
In this section, we prove the case of Theorem \ref{mainhard} where one of the parameters is equal to 1. More formally, we prove the following result:
\begin{Lemma}\label{hard1}
$(k,1)$-BLFD is NP-complete for all $k \in \{9,10,\ldots\}\cup \infty$.
\end{Lemma}
\begin{proof}
We prove Lemma \ref{hard1} by a reduction from $(3,B2)$-SAT. Let $(X,\mathcal{C})$ be an instance of $(3,B2)$-SAT. We first describe a variable gadget $G_x$ for some $x \in X$. First, we let $G_x$ contain a cycle $u_x^1u_x^2w_x^1u_x^3u_x^4w_x^2u_x^1$ on 6 vertices. Next, we add 4 more vertices $v_x^1,\ldots,v_x^4$ and edges $u_x^{i}v_x^{i}$ for $i=1,\ldots,4$. Finally, we add pendant long 1-forcers to $w_x^1$ and $w_x^2$ and we add pendant short 1-forcers to $v_x^{i}$ for $i=1,\ldots,4$. 
For an illustration, see Figure \ref{redu_1}.
\begin{figure}[h!]
    \centering
        \includegraphics[width=.3\textwidth]{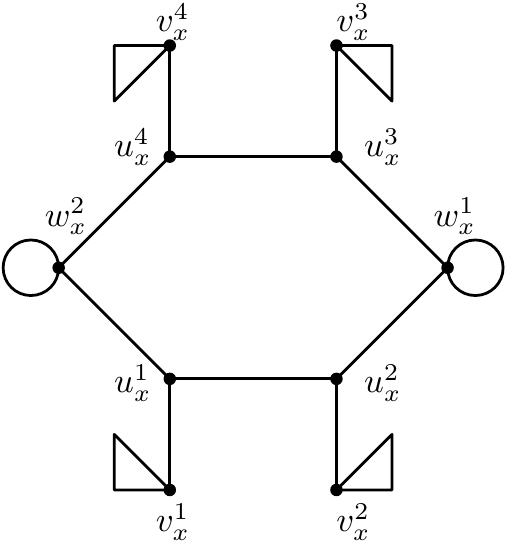}
        \caption{An illustration of a variable gadget $G_x$ for a variable $x$. The cycles indicate pendant long 1-forcers and the triangles indicate pendant short 1-forcers.}\label{redu_1}
\end{figure}
A clause gadget $G_C$ for some $C \in \mathcal{C}$ is a cycle $a_C^1b_C^1a_C^2b_C^2a_C^3b_C^3a_C^1$.

We now create $G$. First, we let $G$ contain a variable gadget $G_x$ for all $x \in X$ and a clause gadget $G_C$ for all $C \in \mathcal{C}$. Now, for every $x \in X$ and $C \in \mathcal{C}$ with $x \in C$, we add an edge linking $\{v_x^1,v_x^2\}$ and $\{a_C^1,a_C^2,a_C^3\}$ to $G$ and for every $x \in X$ and $C \in \mathcal{C}$ with $\bar{x} \in C$, we add an edge linking $\{v_x^3,v_x^4\}$ and $\{a_C^1,a_C^2,a_C^3\}$ to $G$. We do this in a way that we add a perfect matching between $\bigcup_{x \in X}\{v_x^1,\ldots,v_x^4\}$ and $\bigcup_{C \in \mathcal{C}}\{a_C^1,a_C^2,a_C^3\}$. Observe that this is possible because $(X,\mathcal{C})$ is an instance of $(3,B2)$-SAT. This finishes the description of $G$.

For an illustration, see Figure \ref{redu_2}.
\begin{figure}[h!]
    \centering
        \includegraphics[width=.6\textwidth]{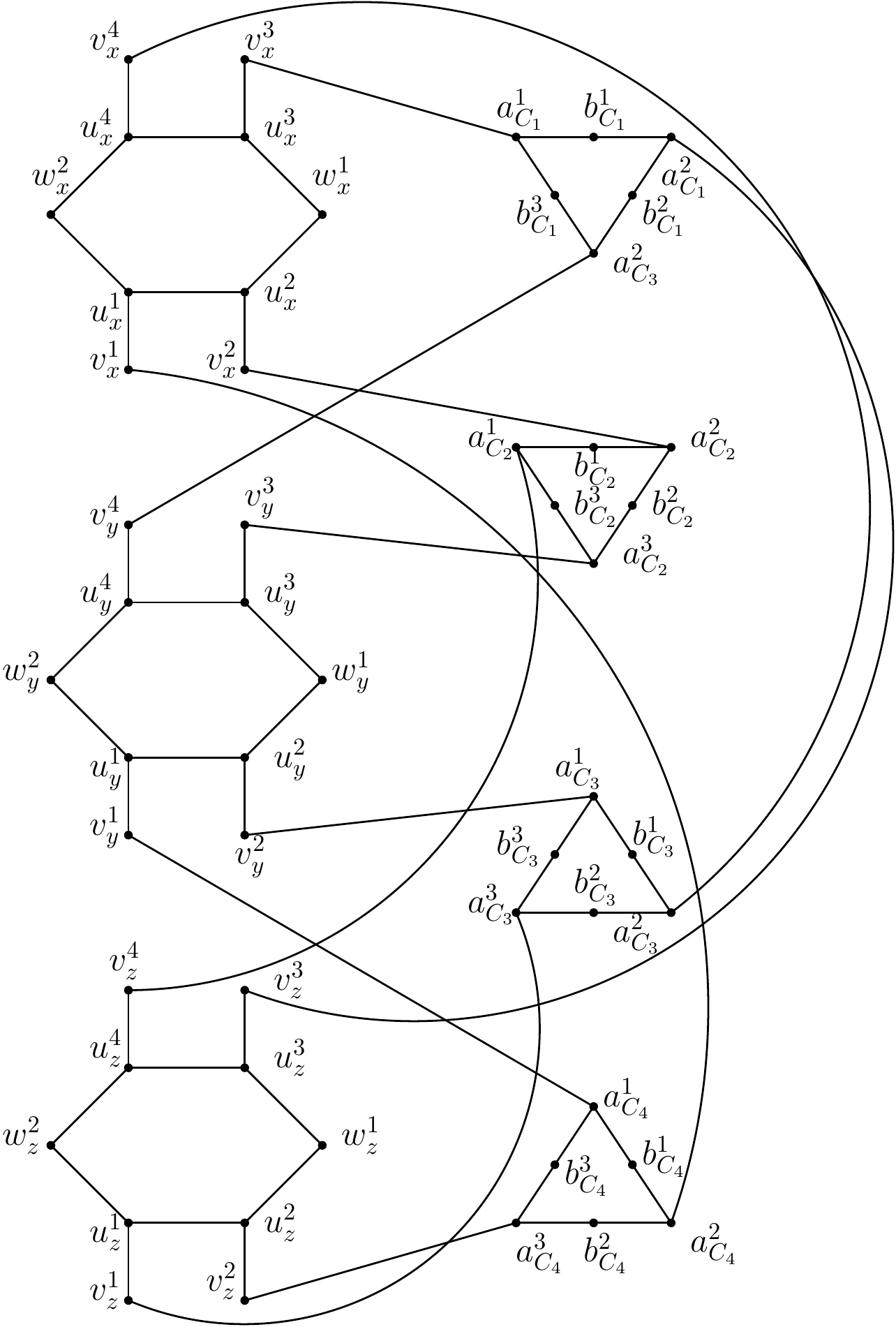}
        \caption{An example for the graph $G$ obtained from the instance $(X,\mathcal{C})$ of $(3,B2)$-SAT where $X=\{x,y,z\}$ and $\mathcal{C}=\{C_1=\{\bar{x},\bar{y},\bar{z}\},C_2=\{x,\bar{y},\bar{z}\},C_3=\{\bar{x},y,z\},C_4=\{x,y,z\}\}$. The short and long 1-forcers have been omitted due to space restrictions.}\label{redu_2}
\end{figure}
\begin{Claim}\label{facsat}
Suppose that $G$ is a yes-instance of $(\infty,1)$-BLFD. Then $(X,\mathcal{C})$ is a yes-instance of $(3,B2)$-SAT.
\end{Claim}
\begin{proof}
Let $(F_\infty,F_1)$ be a decomposition of $G$ into a linear forest and a matching. Consider some $x \in X$. By Proposition \ref{short1}, we obtain that the edges incident to $w_x^1$ and $w_x^2$ which are part of the  long 1-forcers are contained in $E(F_1)$. This yields that $u_x^2w_x^1,w_x^1u_x^3u_x^4w_x^2$ and $w_x^2u_x^1$ are contained in $E(F_\infty)$. As $E(F_\infty)$ cannot contain the cycle $u_x^1u_x^2w_x^1u_x^3u_x^4w_x^2u_x^1$ we obtain that at least one of $u_x^1u_x^2$ and $u_x^3u_x^4$ is not contained in $E(F_\infty)$. We now define a truth assignment $\phi:X \rightarrow \{TRUE,FALSE\}$ in the following way: If $u_x^1u_x^2 \in E(F_\infty)$, we set $\phi(x)$ to $TRUE$ and if $u_x^3 u_x^4 \in E(F_\infty)$, we set $\phi(x)$ to $FALSE$. If none of $u_x^1u_x^2$ and $u_x^3u_x^4$ are contained in $E(F_\infty)$, we set $\phi(x)$ to an arbitrary value.

We now prove that $\phi$ is a satisfying assignment for $(X,\mathcal{C})$. Consider some $C \in \mathcal{C}$. By symmetry, we may suppose that $C$ contains 3 positive literals $x_1,x_2,x_3$ and that $G$ contains the edges $v_{x_i}^1a_C^{i}$ for $i=1,2,3$ . Suppose for the sake of a contradiction that $\phi(x_1)=\phi(x_2)=\phi(x_3)=FALSE$. By construction for $i=1,2,3$, we obtain that $u_{x_i}^1u_{x_i}^2\in E(F_1)$ and hence $u_{x_i}^1v_{x_i}^1\in E(F_\infty)$. Further, by Proposition \ref{short1}, we obtain that $v_{x_i}^1$ is also incident to an edge in $E(F_\infty)$ which is contained in a pendant short 1-forcer. We hence obtain that the edge $v_{x_i}^1a_C^{i}$ is contained in $E(F_1)$ for $i=1,2,3$. We hence obtain that all the edges of the cycle $a_C^1b_C^1a_C^2b_C^2a_C^3b_C^3a_C^1$ are contained in $E(F_\infty)$, a contradiction to $F_\infty$ being a linear forest.
\end{proof}

\begin{Claim}\label{satfac}
Suppose that $(X,\mathcal{C})$ is a yes-instance of $(3,B2)$-SAT. Then $G$ is a yes-instance of $(9,1)$-BLFD.
\end{Claim}
\begin{proof}
We describe a $(9,1)$-bounded linear forest decomposition $(F_9,F_1)$ of $G$.
For all long and short 1-forcers, we distribute their edges to $E(F_9)$ and $E(F_1)$ as described in Propositions \ref{long1} and \ref{short1}, respectively. Next, for all $x \in X$, we let $u_x^2w_x^1,w_x^1u_x^3u_x^4w_x^2$ and $w_x^2u_x^1$ be contained in $E(F_\infty)$. Now let $\phi:X \rightarrow \{TRUE,FALSE\}$ be a satisfying assignment for $(X,\mathcal{C})$. For all $x \in X$ with $\phi(x)=TRUE$, we let $E(F_1)$ contain $u_x^3u_x^4,u_x^1v_x^1,u_x^2v_x^2$ and the two edges linking $\{v_x^3,v_x^4\}$ and $\bigcup_{C \in \mathcal{C}}\{a_C^1,a_C^2,a_C^3\}$. For all $x \in X$ with $\phi(x)=FALSE$, we let $E(F_1)$ contain $u_x^1u_x^2,u_x^3v_x^3,u_x^4v_x^4$ and the two edges linking $\{v_x^1,v_x^2\}$ and $\bigcup_{C \in \mathcal{C}}\{a_C^1,a_C^2,a_C^3\}$.

We now show how to extend $F_1$ to the clause gadgets. Consider some $C \in \mathcal{C}$. As $C$ is satisfied by $\phi$ and construction, we obtain that $|\delta_G(V(G_C))\cap E(F_1)|\leq 2$. If $|\delta_G(V(G_C))\cap E(F_1)|=0$, we let $E(F_1)$ contain $a_C^1b_C^1,a_C^2b_C^2$ and $a_C^3b_C^3$. If $|\delta_G(V(G_C))\cap E(F_1)|=1$, say $\delta_G(V(G_C))\cap E(F_1)$ contains the edge incident to $a_C^1$, we let $E(F_1)$ contain $b_C^1a_C^2$ and $a_C^3b_C^3$.  If $|\delta_G(V(G_C))\cap E(F_1)|=2$, say $\delta_G(V(G_C))\cap E(F_1)$ contains the edges incident to $a_C^1$ and $a_C^2$, we let $E(F_1)$ contain $b_C^2a_C^3$. Finally, we define $F_9$ by $E(F_9)=E(G)-E(F_1)$.
For an illustration, see Figure \ref{redu_3}.
\begin{figure}[h!]
    \centering
        \includegraphics[width=.6\textwidth]{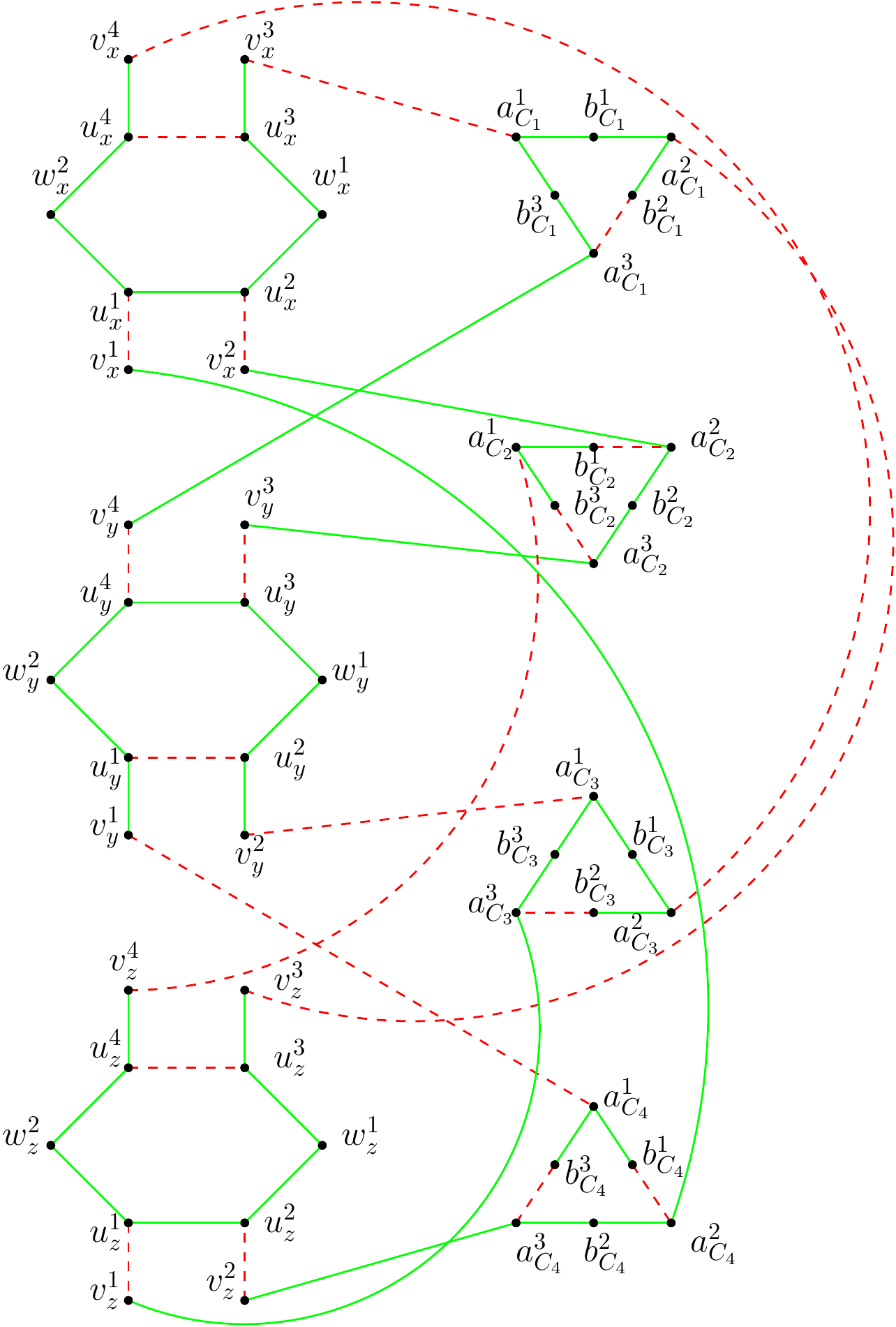}
        \caption{An example for the linear forest decomposition $(F_9,F_1)$ of the graph $G$ depicted in Figure \ref{redu_2} obtained from the instance $(X,\mathcal{C})$ of $(3,B2)$-SAT where $X=\{x,y,z\}$ and $\mathcal{C}=\{C_1=\{\bar{x},\bar{y},\bar{z}\},C_2=\{x,\bar{y},\bar{z}\},C_3=\{\bar{x},y,z\},C_4=\{x,y,z\}\}$ together with the satisfying assignment $\phi$ defined by $\phi(x)=TRUE,\phi(y)=FALSE$ and $\phi(z)=TRUE$. The dashed red edges are in $E(F_1)$ and the solid green edges are in $E(F_9)$. Again, the short and long 1-forcers have been omitted due to space restrictions.}\label{redu_3}
\end{figure}

It is easy to see that $(F_9,F_1)$ is a decomposition of $G$ into a 9-bounded linear forest and a matching. 
\end{proof}
Claims \ref{facsat} and \ref{satfac} imply Lemma \ref{hard1}.
\end{proof}
\subsection{A bounded length linear forest and an arbitrary linear forest}
\label{kinfhard}
Here we give a reduction for the case where one of the forests is unrestricted and the other one is restricted by a parameter which is at least 2. The reduction is similar to the one used in the proof of Lemma \ref{hard1}, but for the sake of readability, we give it separately. Drawings have been omitted due to their similarity with the ones in Section \ref{k1hard}. Concretely, we prove the following result:

\begin{Lemma}\label{hard2}
$(\infty,k)$-BLFD is NP-complete for all integers $k \geq 2$.
\end{Lemma}
\begin{proof}
We prove Lemma \ref{hard2} by a reduction from $(3,B2)$-SAT. Let $(X,\mathcal{C})$ be an instance of $(3,B2)$-SAT. We first describe a variable gadget $G_x$ for some $x \in X$. First, we let $G_x$ contain a cycle $u_x^1u_x^2w_x^1u_x^3u_x^4w_x^2u_x^1$ on 6 vertices. Next, we add 4 more vertices $v_x^1,\ldots,v_x^4$ and edges $u_x^{i}v_x^{i}$ for $i=1,\ldots,4$. Finally, we add pendant long $(\infty,k)$-forcers to $w_x^1$ and $w_x^2$ and we add a pendant short $(k+1,k)$-forcer and a pendant $(\infty,k)$-path forcer to $v_x^{i}$ for $i=1,\ldots,4$. A clause gadget $G_C$ for some $C \in \mathcal{C}$ is a cycle $a_C^1b_C^1a_C^2b_C^2a_C^3b_C^3a_C^1$.

We now create $G$. First, we let $G$ contain a variable gadget $G_x$ for all $x \in X$ and a clause gadget $G_C$ for all $C \in \mathcal{C}$. Now, for every $x \in X$ and $C \in \mathcal{C}$ with $x \in C$, we add an edge linking $\{v_x^1,v_x^2\}$ and $\{a_C^1,a_C^2,a_C^3\}$ to $G$ and for every $x \in X$ and $C \in \mathcal{C}$ with $\bar{x} \in C$, we add an edge linking $\{v_x^3,v_x^4\}$ and $\{a_C^1,a_C^2,a_C^3\}$ to $G$. We do this in a way that we add a perfect matching between $\bigcup_{x \in X}\{v_x^1,\ldots,v_x^4\}$ and $\bigcup_{C \in \mathcal{C}}\{a_C^1,a_C^2,a_C^3\}$. Observe that this is possible because $(X,\mathcal{C})$ is an instance of $(3,B2)$-SAT. This finishes the description of $G$. We show in the following that $G$ is a yes-instance of $(\infty,k)$-BLFD if and only if $(X,\mathcal{C})$ is a yes-instance of $(3,B2)$-SAT.

First suppose that $G$ is a yes-instance of $(\infty,k)$-BLFD.
Let $(F_\infty,F_k)$ be a $(\infty,k)$-bounded linear forest decomposition of $G$. Consider some $x \in X$. By Proposition \ref{longinf}, we obtain that the edges incident to $w_x^1$ and $w_x^2$ which are part of the long $(\infty,k)$-forcer are contained in $E(F_k)$. This yields that $u_x^2w_x^1,w_x^1u_x^3u_x^4w_x^2$ and $w_x^2u_x^1$ are contained in $E(F_\infty)$. As $F_\infty$ cannot contain the cycle $u_x^1u_x^2w_x^1u_x^3u_x^4w_x^2u_x^1$ we obtain that at least one of $u_x^1u_x^2$ and $u_x^3u_x^4$ is not contained in $E(F_\infty)$. We now define a truth assignment $\phi:X \rightarrow \{TRUE,FALSE\}$ in the following way: If $u_x^1u_x^2 \in E(F_\infty)$, we set $\phi(x)$ to $TRUE$ and if $u_x^3 u_x^4 \in E(F_\infty)$, we set $\phi(x)$ to $FALSE$. If none of $u_x^1u_x^2$ and $u_x^3u_x^4$ are contained in $E(F_\infty)$, we set $\phi(x)$ to an arbitrary value.

We now prove that $\phi$ is a satisfying assignment for $(X,\mathcal{C})$. Consider some $C \in \mathcal{C}$. By symmetry, we may suppose that $C$ contains 3 positive literals $x_1,x_2,x_3$ and that $G$ contains the edges $v_{x_i}^1a_C^{i}$ for $i=1,2,3$ . Suppose for the sake of a contradiction that $\phi(x_1)=\phi(x_2)=\phi(x_3)=FALSE$. By construction for $i=1,2,3$, we obtain that $u_{x_i}^1u_{x_i}^2\in E(F_k)$. By Proposition \ref{path}, there is a path of length $k-1$ that is entirely contained in the $(\infty,k)$-path forcer pendant at $v_{x_i}^1$ in $F_k$. This yields that if $u_{x_i}^1v_{x_i}^1\in E(F_k)$, then $F_k$ contains a path of length $k+1$, a contradiction.  We hence obtain $u_{x_i}^1v_{x_i}^1\in E(F_\infty)$. Further, by a similar argument to the one in Proposition \ref{short}, we obtain that $v_{x_i}^1$ is also incident to an edge in $E(F_\infty)$ which is contained in the pendant short $(k+1,k)$-forcer. As $F_k$ contains a path of length $k-1$ entirely contained in the pendant $(\infty,k)$-path forcer, we obtain that the edge $v_{x_i}^1a_C^{i}$ is the last edge of a path of length in $F_k$ which is disjoint from the clause gadget $G_C$ for $i=1,2,3$. We hence obtain that all the edges of the cycle $a_C^1b_C^1a_C^2b_C^2a_C^3b_C^3a_C^1$ are contained in $E(F_\infty)$, a contradiction to $F_\infty$ being a linear forest. Hence $\phi$ satisfies $(X,\mathcal{C})$ and so $(X,\mathcal{C})$ is a yes-instance of $(3,B2)$-SAT.

Now suppose that $(X,\mathcal{C})$ is a yes-instance of $(3,B2)$-SAT.
 Let $\phi:X \rightarrow \{TRUE,FALSE\}$ be a satisfying assignment for $(X,\mathcal{C})$. We now describe an $(\infty,k)$-bounded linear forest decomposition $(F_\infty,F_k)$ of $G$. For all long $(\infty,k)$-forcers, $(\infty,k)$-path forcers and short $(k+1,k)$-forcers, we let $E(F_k)$ contain the edges corresponding to the decompositions described in Propositions \ref{longinf}, \ref{path} and \ref{short}, respectively. For all $x \in X$ with $\phi(x)=TRUE$, we add to $E(F_k)$ the edges $u_x^3u_x^4,u_x^1v_x^1,u_x^2v_x^2$ and the two edges linking $\{v_x^3,v_x^4\}$ and $\bigcup_{C \in \mathcal{C}}\{a_C^1,a_C^2,a_C^3\}$. For all $x \in X$ with $\phi(x)=FALSE$, we add to $E(F_k)$ the edges $u_x^1u_x^2,u_x^3v_x^3,u_x^4v_x^4$ and the two edges linking $\{v_x^1,v_x^2\}$ and $\bigcup_{C \in \mathcal{C}}\{a_C^1,a_C^2,a_C^3\}$.

We now show how to extend $F_k$ to the clause gadgets. Consider some $C \in \mathcal{C}$. As $C$ is satisfied by $\phi$ and construction, we obtain that $|\delta_G(V(G_C))\cap E(F_k)|\leq 2$. If $|\delta_G(V(G_C))\cap E(F_k)|=0$, we let $E(F_1)$ contain $a_C^1b_C^1,a_C^2b_C^2$ and $a_C^3b_C^3$. If $|\delta_G(V(G_C))\cap E(F_k)|=1$, say $\delta_G(V(G_C))\cap E(F_k)$ contains the edge incident to $a_C^1$, we let $E(F_k)$ contain $b_C^1a_C^2$ and $a_C^3b_C^3$.  If $|\delta_G(V(G_C))\cap E(F_k)|=2$, say $\delta_G(V(G_C))\cap E(F_1)$ contains the edges incident to $a_C^1$ and $a_C^2$, we let $E(F_1)$ contain $b_C^2a_C^3$. Finally, we define $F_\infty$ by $E(F_\infty)=E(G)-E(F_k)$. It is easy to see that $(F_\infty,F_k)$ is an $(\infty,k)$-bounded linear forest decomposition of $G$. 

Hence $G$ is a yes-instance of $(\infty,k)$-BLFD, so the statement follows.

\end{proof}

\subsection{Two bounded length linear forests}\label{klhard}
In this section, we prove the hardness in the case that both $k$ and $\ell$ are positive integers strictly greater than 1. More concretely, we prove the following result.
\begin{Lemma}\label{hard3}
For any pair of integers $k,\ell \geq 2$, it is NP-complete to decide whether a given graph has a $(k,\ell)$-bounded linear forest decomposition.
\end{Lemma}

The proof is split into two parts. In Section \ref{infgad}, we introduce a more involved gadget we need for the main reduction. In Section \ref{mainred}, we give the main proof of Lemma \ref{hard3}.
\subsubsection{A more involved gadget}\label{infgad}
Let $G$ be a graph and $A \subseteq V(G)$ with $d_G(a)=1$ for all $a \in A$. For a positive integer $k$ and a $k$-bounded linear forest $F$ of $G$, we say that $A$ is {\it $(k,F)$-covered} if every $a \in A$ is the endpoint of a path of length $k$ in $F$. For some integers $\alpha,k,\ell$ with $k \geq \ell \geq 2$, an {\it$(\alpha,k,\ell)$-gadget} is a graph $G$ together with a set $A \subseteq V(G)$ with the following properties:
\begin{itemize}
\item $d_G(a)=1$ for all $a \in A$,
\item $|A|=\alpha$,
\item there is a $(k,\ell)$-bounded linear forest decomposition $(F,F')$ of $G$ such that $A$ is $(k,F)$-covered,
\item there is a $(k,\ell)$-bounded linear forest decomposition $(F,F')$ of $G$ such that $A$ is $(\ell,F')$-covered,
\item for every $(k,\ell)$-bounded linear forest decomposition $(F,F')$ of $G$, either $A$ is $(k,F)$-covered or $A$ is $(\ell,F')$-covered.
\end{itemize}
\begin{Lemma}\label{exists_gadget}
For all positive integers $\alpha,k,\ell$ with $k \geq \ell \geq 2$ and $\alpha \geq 2$, there exists an $(\alpha,k,\ell)$-gadget whose size is linear in $\alpha$.
\end{Lemma}
\begin{proof}
We need to distinguish two different cases.
\begin{casefirst}\label{ungleich}
$k>\ell$.
\end{casefirst}
We fix some $k>\ell\geq 2$ and $\alpha \geq 2$ and construct an $(\alpha,k,\ell)$-gadget $G$. First, we let $V(G)$ contain $4\alpha$ vertices $v_1,\ldots,v_\alpha,w_1,\ldots,w_\alpha,a_1,\ldots,a_\alpha,b_1,\ldots,b_\alpha$ where the indices $1,\ldots,\alpha$ are considered to be elements of the cyclic additive group on $\alpha$ elements. Next, for $i=1,\ldots,\alpha$, we add the edges $v_ia_i$ and $w_ib_i$. Further, for $i=1,\ldots,\alpha$, we join $v_i$ and $w_i$ by a path $P_i$ of length $k-1$ and identify each of the interior vertices of this path with the tip vertex of a long $(k,\ell)$-forcer.  Finally, for $i=1,\ldots,\alpha$, we join $w_i$ and $v_{i+1}$ by a path $Q_i$ of length $\ell-1$ and identify each of the interior vertices of this path with the tip vertex of a short $(k,\ell)$-forcer. This finishes the description of $G$. For an illustration, see Figure \ref{image_7}. 
\begin{figure}[h!]
    \centering
        \includegraphics[width=.5\textwidth]{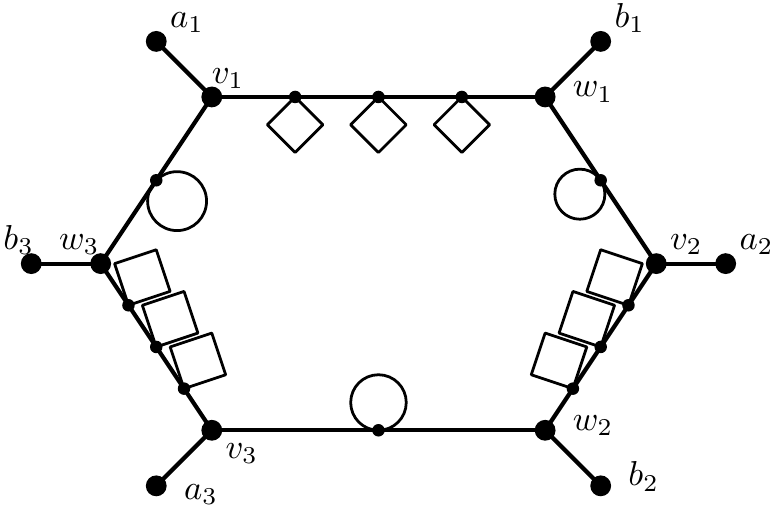}
        \caption{A $(3,5,3)$-gadget. The squares indicate attached long $(5,3)$-forcers and the cycles indicate attached short $(5,3)$-forcers.}\label{image_7}
\end{figure}
Let $A=\{a_1,\ldots,a_\alpha\}$ and observe that the size of $G$ is linear in $\alpha$. The following claims show that $(G,A)$ is an $(\alpha,k,\ell)$-gadget.
\begin{Claim}
There is a $(k,\ell)$-bounded linear forest decomposition $(F_k,F_\ell)$ of $G$ such that $A$ is $(k,F_k)$-covered.
\end{Claim}
\begin{proof}
For all $i=1,\ldots,\alpha$, let $E(F_k)$ contain $v_ia_i$ and $E(P_i)$ and let $E(F_\ell)$ contain $w_ib_i$ and $E(Q_i)$. Further, we extend this decomposition to the short and long $(k,\ell)$-forcers by their unique decompositions described in Propositions \ref{short} and \ref{long}. It follows immediately from the construction that $(F_k,F_\ell)$  is a $(k,\ell)$-bounded linear forest decomposition of $G$ and that $A$ is $(k,F_k)$-covered. For an illustration, see Figure \ref{image_8}. 
\begin{figure}[h!]
    \centering
        \includegraphics[width=.5\textwidth]{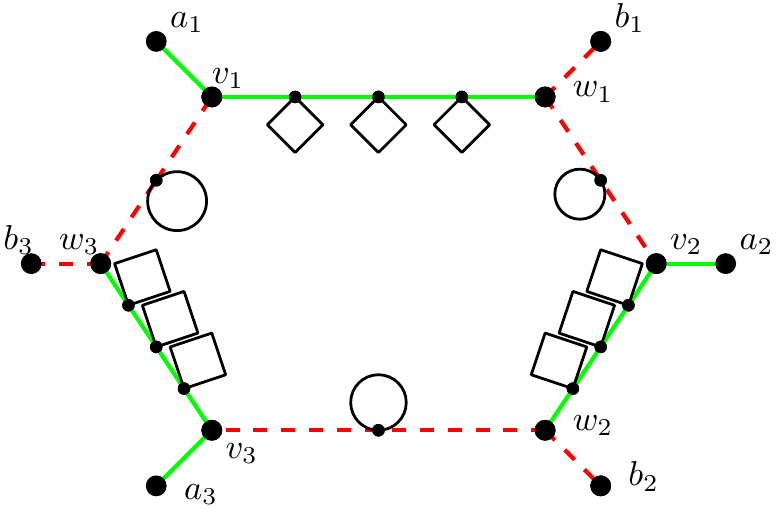}
        \caption{A $(5,3)$-bounded linear forest decomposition $(F_5,F_3)$ of a $(3,5,3)$-gadget in which $A$ is $(5,F_5)$-covered. The dashed red edges are in $E(F_3)$ and the solid green edges are in $E(F_5)$. The edges inside the short and long $(5,3)$-forcers have been omitted. They correspond to the ones in Figures \ref{image_2} and \ref{image_4}, respectively.}\label{image_8}
\end{figure}
\end{proof}
\begin{Claim}
There is a $(k,\ell)$-bounded linear forest decomposition $(F_k,F_\ell)$ of $G$ such that $A$ is $(\ell,F_\ell)$-covered.
\end{Claim}
\begin{proof}
For all $i=1,\ldots,\alpha$, let $E(F_k)$ contain $w_ib_i$ and $E(P_i)$ and let $E(F_\ell)$ contain $v_ia_i$ and $E(Q_i)$. Further, we extend this decomposition to the short and long $(k,\ell)$-forcers by their unique decompositions described in Propositions \ref{short} and \ref{long}. It follows immediately from the construction that $(F_k,F_\ell)$  is a $(k,\ell)$-bounded linear forest decomposition of $G$ and that $A$ is $(\ell,F_\ell)$-covered. For an illustration, see Figure \ref{image_9}. 
\begin{figure}[h!]
    \centering
        \includegraphics[width=.5\textwidth]{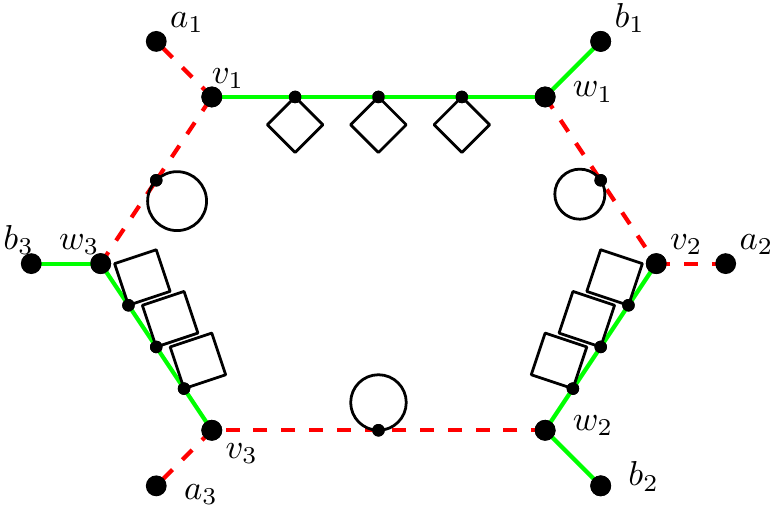}
        \caption{A $(5,3)$-bounded linear forest decomposition $(F_5,F_3)$ of a $(3,5,3)$-gadget in which $A$ is $(3,F_3)$-covered. The dashed red edges are in $E(F_3)$ and the solid green edges are in $E(F_5)$. The edges inside the short and long $(5,3)$-forcers have been omitted. They correspond to the ones in Figures \ref{image_2} and \ref{image_4}, respectively.}\label{image_9}
\end{figure}
\end{proof}
\begin{Claim}
For every $(k,\ell)$-bounded linear forest decomposition $(F_k,F_\ell)$ of $G$, either $A$ is $(k,F_k)$-covered or $A$ is $(\ell,F_\ell)$-covered.
\end{Claim}
\begin{proof}
Let $(F_k,F_\ell)$ be a $(k,\ell)$-bounded linear forest decomposition of $G$. For $i=1,\ldots,\alpha$, by Proposition \ref{long}, we obtain that every interior vertex of $P_i$ is incident to two edges of $E(F_\ell)$ which are contained in a long $(k,\ell)$-forcer. This yields that both the edges of $P_i$ which are incident to this vertex are contained in $E(F_k)$. This yields that $E(P_i)\subseteq E(F_k)$ for $i=1,\ldots,\alpha$.   Next, if $\ell\geq 3$, for $i=1,\ldots,\alpha$, by Proposition \ref{short}, we obtain that every interior vertex of $Q_i$ is the last vertex of a path of length $k$ in $F_k$ which is contained in a short $(k,\ell)$-forcer. This yields that both the edges of $Q_i$ which are incident to this vertex are contained in $E(F_\ell)$. This yields that $E(Q_i)\subseteq E(F_\ell)$ for $i=1,\ldots,\alpha$. If $\ell=2$, then for $i=1,\ldots,\alpha$, the unique edge of $Q_i$ cannot be contained in $E(F_k)$, as this would force $F_k$ to contain a $v_iw_{i+1}$-path of length $2k-1>k$. In either case, we obtain that $E(P_i)\subseteq E(F_k)$ and $E(Q_i)\subseteq E(F_\ell)$ for $i=1,\ldots,\alpha$.

Now suppose that $v_1a_1 \in E(F_k)$. We will show that $v_ia_i \in E(F_k)$ for all $i=1,\ldots,\alpha$. Suppose that this is the case for all integers up to some fixed $i$. Inductively, it suffices to prove that the statement holds for $i+1$. Observe that $F_k$ contains an $a_iw_i$-path of length $k$. This yields that $w_ib_i \in E(F_\ell)$. Hence $F_\ell$ contains a $b_iv_{i+1}$-path of length $\ell$. It follows that $v_{i+1}a_{i+1}\in E(F_k)$. We hence obtain that $v_ia_i \in E(F_k)$ for all $i=1,\ldots,\alpha$. As $P_i\subseteq E(F_k)$ for $i=1,\ldots,\alpha$, we obtain that $a_i$ is the last vertex of a path of length $k$ in $F_k$, hence $A$ is $(k,F_k)$-covered. A similar argument shows that if $v_1a_1 \in E(F_\ell)$, then $A$ is $(\ell,F_\ell)$-covered.
\end{proof}
This finishes Case \ref{ungleich}. We now consider the remaining case.

\begin{casefirst}\label{gleich}
$k=\ell$.
\end{casefirst}
The gadget for Case \ref{gleich} is somewhat similar to the one for Case \ref{ungleich}, but in order to improve readability, we treat this case separately. 
We fix some $k\geq 2$ and $\alpha \geq 2$ and construct an $(\alpha,k,k)$-gadget $G$. First, we let $V(G)$ contain $4\alpha$ vertices $v_1,\ldots,v_\alpha,w_1,\ldots,w_\alpha, \\ a_1,\ldots,a_\alpha,b_1,\ldots,b_\alpha$ where the indices $1,\ldots,\alpha$ are considered to be elements of the cyclic additive group on $\alpha$ elements. Next, for $i=1,\ldots,\alpha$, we add the edges $v_ia_i$ and $w_ib_i$. Further, for $i=1,\ldots,\alpha$, we join $v_i$ and $w_i$ by a path $P_i$ of length $k-1$ and identify each of the interior vertices of this path with the tip vertex of a symmetric $k$-forcer.  Finally, for $i=1,\ldots,\alpha$, we join $w_i$ and $v_{i+1}$ by a path $Q_i$ of length $\ell-1$ and identify each of the interior vertices of this path with the tip vertex of a symmetric $k$-forcer. This finishes the description of $G$. For an illustration, see Figure \ref{image_10}. 
\begin{figure}[h!]
    \centering
        \includegraphics[width=.5\textwidth]{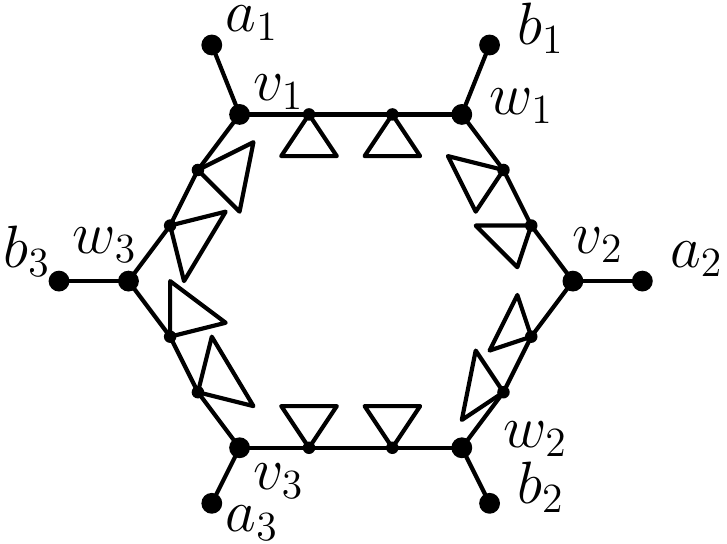}
        \caption{A $(3,4,4)$-gadget. The triangles indicate attached symmetric 4-forcers.}\label{image_10}
\end{figure}

Let $A=\{a_1,\ldots,a_\alpha\}$ and observe that the size of $G$ is linear in $\alpha$. The following claims show that $(G,A)$ is an $(\alpha,k,k)$-gadget indeed.
\begin{Claim}
There is a $(k,k)$-bounded linear forest decomposition $(F,F')$ of $G$ such that $A$ is $(k,F)$-covered.
\end{Claim}
\begin{proof}
For all $i=1,\ldots,\alpha$, let $E(F)$ contain $v_ia_i$ and $E(P_i)$ and let $E(F')$ contain $w_ib_i$ and $E(Q_i)$. Further, we extend this decomposition to the symmetric $k$-forcers by their unique decompositions described in Proposition \ref{symmetric} for which every interior vertex of $P_i$ is incident to an edge of $F'$ and every interior vertex of $Q_i$ is incident to an edge of $F$ in the attached symmetric $k$-forcer for $i=1,\ldots,\alpha$. It follows immediately from the construction that $(F,F')$  is a $(k,k)$-bounded linear forest decomposition of $G$ and that $A$ is $(k,F)$-covered. For an illustration, see Figure \ref{image_11}. 

\begin{figure}[h!]
    \centering
        \includegraphics[width=.5\textwidth]{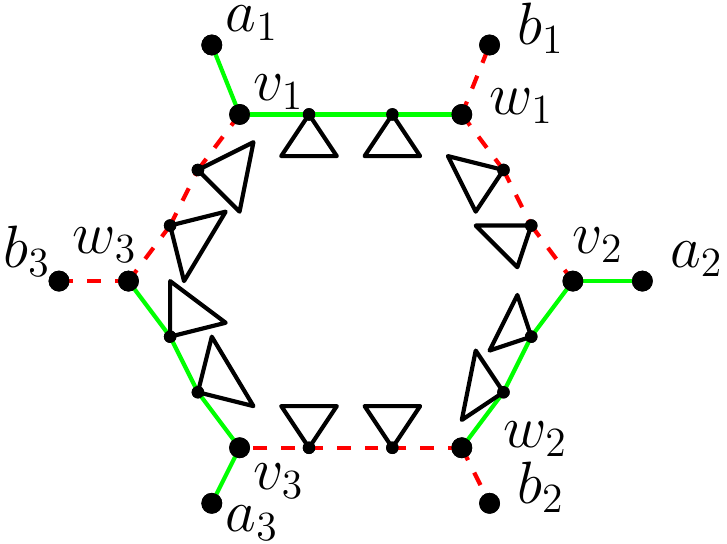}
        \caption{A $(4,4)$-bounded linear forest decomposition $(F,F')$ of a $(3,4,4)$-gadget in which $A$ is $(4,F)$-covered. The dashed red edges are in $E(F')$ and the solid green edges are in $E(F)$. The edges inside the symmetric $4$-forcers have been omitted. They correspond to the one in Figure \ref{image_6}.}\label{image_11}
\end{figure}
\end{proof}
\begin{Claim}
For every $(k,k)$-bounded linear forest decomposition $(F,F')$ of $G$, either $A$ is $(k,F)$-covered or $A$ is $(k,F')$-covered.
\end{Claim}
\begin{proof}
Let $(F,F')$ be a $(k,k)$-bounded linear forest decomposition of $G$. Let first $v$ be an interior vertex of $P_i$ for some $i \in \{1,\ldots,\alpha\}$. By Proposition \ref{symmetric}, we obtain that $v$ is the last vertex of a path of length $k$ in $F$ or $F'$ that is entirely contained in the symmetric $k$-forcer attached to $v$. We obtain that the two edges of $P_i$ that are incident  to $v$ are either both contained in $E(F)$ or both contained in $E(F')$. This yields that either $E(P_i)\subseteq E(F)$ or $E(P_i)\subseteq E(F')$ holds. Similarly, we obtain that for $i=1,\ldots,\alpha$, one of $E(Q_i)\subseteq E(F)$ and $E(Q_i)\subseteq E(F')$ holds.

By symmetry, we may suppose that $E(P_1)\subseteq E(F)$ holds. Further, suppose that $v_1a_1\in E(F)$ holds. We will show that $v_ia_i \in E(F)$ and $E(P_i)\subseteq E(F)$ hold for all $i=1,\ldots,\alpha$. Suppose that this is the case for all integers up to some fixed $i$. Inductively, it suffices to prove that the statement holds for $i+1$. Observe that $F$ contains an $a_iw_i$-path of length $k$. This yields that $w_ib_i \in E(F')$ and $E(Q_i)\subseteq E(F')$. Hence $F'$ contains a $b_iv_{i+1}$-path of length $k$. It follows that $v_{i+1}a_{i+1}\in E(F)$ and $E(P_{i+1})\subseteq E(F)$. We hence obtain that $v_ia_i \in E(F)$ and $E(P_i)\subseteq E(F)$ hold for all $i=1,\ldots,\alpha$. As $P_i\subseteq E(F)$ for $i=1,\ldots,\alpha$, we obtain that $a_i$ is the last vertex of a path of length $k$ in $F$, hence $A$ is $(k,F)$-covered. A similar argument shows that if $v_1a_1 \in E(F')$, then $A$ is $(k,F')$-covered.
\end{proof}
\end{proof}
\subsubsection{The main reduction}\label{mainred}
We are now ready to give the main proof of Lemma \ref{hard3}.
\begin{proof}(of Lemma \ref{hard3})
We prove this by a reduction from MNAE3SAT . Let $(X,\mathcal{C})$ be an instance of MNAE3SAT. We now create an instance $G$ of $(k,\ell)$-BLFD in the following way: For every $x \in X$, we let $G$ contain an $(\alpha_x,k,\ell)$-gadget $(G_x,A_x)$ where $\alpha_x$ is the number of occurences of $x$ in $\mathcal{C}$. Observe that this gadget always exists by Lemma \ref{exists_gadget}. Next for every $C \in \mathcal{C}$, we let $G$ contain a vertex $v_C$. We now add an edge $v_C$ to a vertex in $A_x$ whenever $x$ is contained in $C$. We choose these edges in a way that every vertex in $A_x$ is incident to exactly one edge that is not contained in $E(G_x)$ for all $x \in X$. This finishes the description of $G$. For an illustration, see Figure \ref{redu_4}.
\begin{figure}[h!]
    \centering
        \includegraphics[width=.6\textwidth]{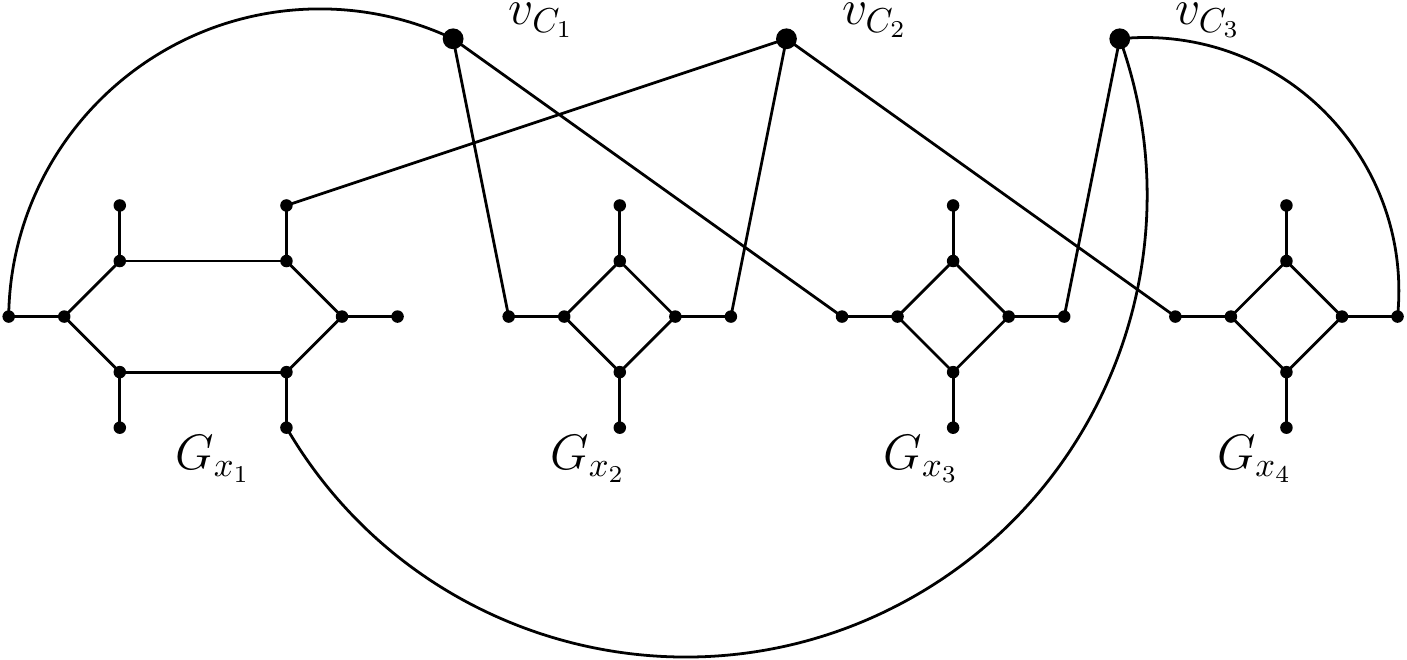}
        \caption{An example for the construction of $G$ in Lemma \ref{hard3} where $k=\ell=2,X=\{x_1,\ldots,x_4\}$ and $\mathcal{C}=\{C_1=\{x_1,x_2,x_3\},C_2=\{x_1,x_2,x_4\},C_3=\{x_1,x_3,x_4\}\}$.}\label{redu_4}
\end{figure}

By the second part of Lemma \ref{exists_gadget}, we have that the size of $G$ is polynomial in the size of $(X,\mathcal{C})$. We now show that $G$ is a yes-instance of $(k,\ell)$-BLFD if and only if $(X,\mathcal{C})$ is a yes-instance of MNAE3SAT.

First suppose that $G$ is a yes-instance of $(k,\ell)$-BLFD, so there is a $(k,\ell)$-bounded linear forest decomposition $(F,F')$ of $G$. For all $x \in X$, as $G_x$ is a $(\alpha_x,k,\ell)$-gadget, we obtain that either $A_x$ is $(k,F)$-covered or $A_x$ is $(\ell,F')$-covered. We now define a truth assignment $\phi:X \rightarrow \{TRUE,FALSE\}$ in the following way: We set $\phi(x)=TRUE$ if $A_x$ is $(k,F)$-covered and we set $\phi(x)=FALSE$ if $A_x$ is $(\ell,F')$-covered. In order to show that $\phi$ is satisfying indeed, consider some $C \in \mathcal{C}$. As $v_C$ is of degree 3 in $G$ and $F$ and $F'$ are linear forests, we obtain that there are variables $x,y \in C$ such that $E(F)$ contains an edge linking $A_x$ and $v_C$ and $E(F')$ contains an edge linking $A_y$ and $v_C$. As $F$ is a $k$-bounded linear forest, we obtain that $A_x$ cannot be $(k,\bar{F})$-covered where $\bar{F}$ is the restriction  of $F$ to $G_x$. It follows that $A_x$ is $(\ell, F')$-covered, so $\phi(x)=FALSE$. Similarly, we obtain that $\phi(y)=TRUE$, so $C$ is satisfied. As $C \in \mathcal{C}$ was chosen arbitrarily, we obtain that $(X,\mathcal{C})$ is satisfied by $\phi$, so $(X,\mathcal{C})$ is a yes-instance of MNAE3SAT.

Now suppose that $(X,\mathcal{C})$ is a yes-instance of MNAE3SAT, so there is a satisfying assignment $\phi:X \rightarrow \{TRUE,FALSE\}$ for $(X,\mathcal{C})$. For every $x \in X$, as $G_x$ is an $(\alpha_x,k,\ell)$-gadget, we can now choose a $(k,\ell)$-bounded linear forest decomposition $(F_x,F'_x)$ of $G_x$ such that $A_x$ is $(k,F_x)$-covered if $\phi(x)=TRUE$ and  $(\ell,F'_x)$-covered if $\phi(x)=FALSE$. We now create $F$ by adding all the edges leaving $G_x$ for a variable $x \in X$ with $\phi(x)=FALSE$ to $\bigcup_{x \in X} E(F_x)$ and we  create $F'$ by adding all the edges leaving $G_x$ for a variable $x \in X$ with $\phi(x)=TRUE$ to $\bigcup_{x \in X} E(F'_x)$.

For an illustration, see Figure \ref{redu_5}.
\begin{figure}[h!]
    \centering
        \includegraphics[width=.6\textwidth]{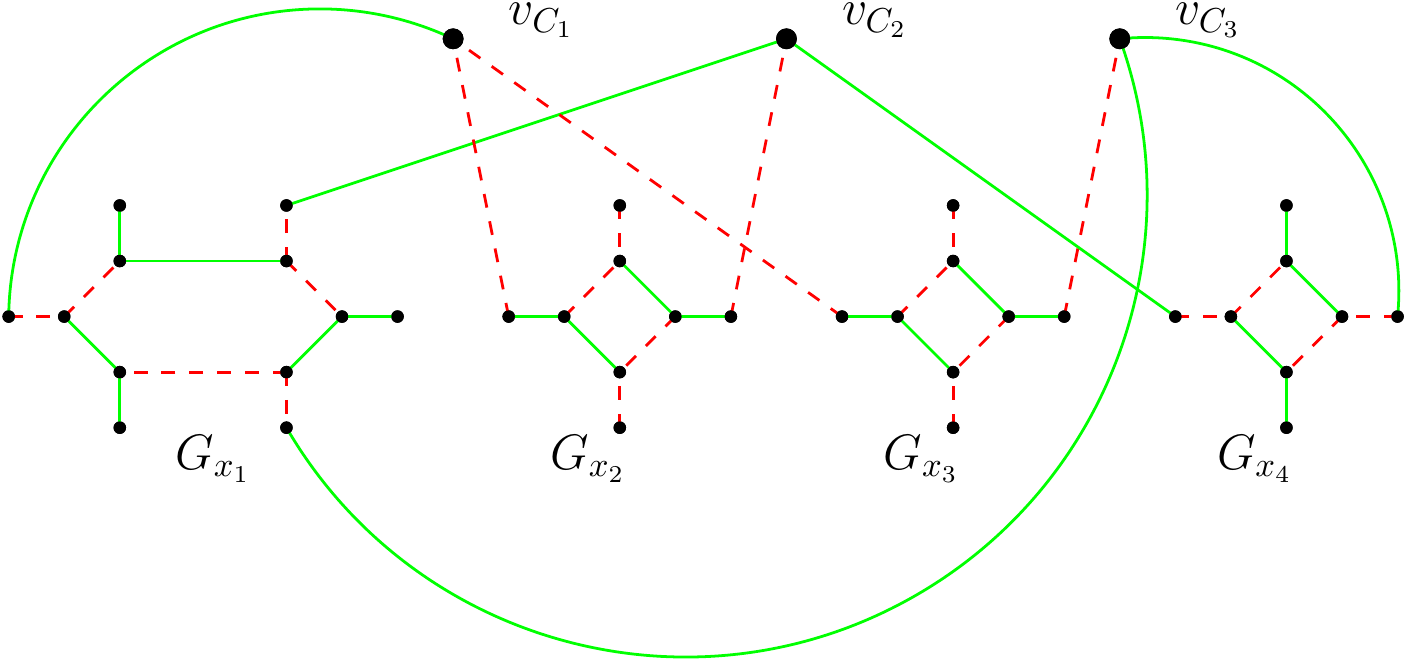}
        \caption{An example for the construction of $(F,F')$ for the instance described in Figure \ref{redu_4} together with the satisfying assignment $\phi$ defined by $\phi(x_1)=\phi(x_4)=TRUE$ and $\phi(x_2)=\phi(x_3)=FALSE$.}\label{redu_5}
\end{figure}

Observe that every component of $F$ or $F'$ is either completely contained in $G_x$ for some $x \in X$ or its edge set is disjoint from $E(G_x)$ for all $x \in X$. Further observe that, as $\phi$ satisfies $(X,\mathcal{C})$, we have that $v_C$ is incident to at least one edge of $F$ and at least one edge of $F'$ for all $C \in \mathcal{C}$. We hence obtain that every component of $F$ or $F'$ is either a component of $F_x$ or $F'_x$ for some $x \in X$ or is a path of length at most 2. As $(F_x,F'_x)$ is a $(k,\ell)$-bounded linear forest decomposition of $G_x$ for all $x \in X$ and $k,\ell \geq 2$, we obtain that $(F,F')$ is a $(k,\ell)$-bounded linear forest decomposition of $G$. Hence $G$ is a yes-instance of $(k,\ell)$-BLFD.

This finishes the proof.
\end{proof}

Finally observe that Lemmas \ref{hard1},\ref{hard2} and \ref{hard3} together with Theorem \ref{inftyinfty} imply Theorem \ref{mainhard}.

\section{Planar graphs of girth 9}\label{gir9}
This section is dedicated to proving Theorem \ref{girth9} using the discharging method. In Section \ref{premgir9}, we give some preliminary results. In Section \ref{propmincou}, we give some structural properties of a minimum counterexample applying the results from Section \ref{premgir9}. Finally, in Section \ref{disch}, we proceed to the discharging procedure which shows no minimum counterexample exists.
\subsection{Preliminaries}\label{premgir9}
For a graph $G$ and $u,v \in V(G)$, we use $dist_G(u,v)$ to denote the number of edges in a shortest $uv$-path in $G$. For a non-degenerate  face $F$ (i.e. a face which is bounded by a cycle) of an embedding of a planar graph $G$, we use $V(F)$ for the vertices of $G$ which are incident to $F$ and $E(F)$ for the edges on the boundary of $F$. For an embedding $G$ of a planar graph, we let $F(G)$ denote the set of faces of the embedding, and for $F \in F(G)$, we let $|F|$ denote the length of the boundary walk of $F$. We say that an ordering of $V(F)$ is {\it canonical} if it follows the order in which the vertices appear on the boundary of the face. Observe that there are $2|V(F)|$ canonical orderings.

\begin{Proposition}\label{wege}
Let $G$ be a planar graph, $F$ a non-degenerate face of an embedding of $G$ and $v_1,\ldots,v_4 \in V(F)$ that appear in this order in a canonical ordering of $V(F)$ such that $\max\{dist_{G-E(F)}(v_1,v_3),dist_{G-E(F)}(v_2,v_4)\} \leq 6$. Then $\min\{dist_{G-E(F)}(v_1,v_2),dist_{G-E(F)}(v_3,v_4)\} \leq 6$.
\end{Proposition}
\begin{proof}
Let $P_1$ be a shortest $v_1v_3$-path in $G-E(F)$ and  $P_2$ be a shortest $v_2v_4$-path in $G-E(F)$. As $G$ is planar, there is some $x \in V(P_1)\cap V(P_2)$. We obtain 
\begin{align*}
6+6 &\geq dist_{G-E(F)}(v_1,v_3)+dist_{G-E(F)}(v_2,v_4)\\
&= dist_{P_1}(v_1,x)+dist_{P_1}(v_3,x)+dist_{P_2}(v_2,x)+dist_{P_2}(v_4,x)\\
&= (dist_{P_1}(v_1,x)+dist_{P_2}(v_2,x))+(dist_{P_1}(v_3,x)+dist_{P_2}(v_4,x))\\
&\geq dist_{G-E(F)}(v_1,v_2)+dist_{G-E(F)}(v_3,v_4),\\
\end{align*}
so the statement follows.
\end{proof}
The following is an immediate consequence of the definition of linear forests.
\begin{Proposition}\label{3paths}
Let $L$ be a linear forest, $v_1,v_2,v_3 \in V(L)$ such that $d_L(v_i)=1$ for $i=1,2,3$ and let $P_i$ be the connected component of $L$ containing $v_i$. Then at least one of $P_1\neq P_2$ and $P_2 \neq P_3$ holds.
\end{Proposition}

The following so-called ``face-charging" inequality is well known (see \cite{CRANSTON2017766} Proposition 3.1):

\begin{Proposition}
\label{facechargingobs}
For any embedding $G_0$ of a planar graph $G$, the following equality holds
\[\sum_{f \in F(G_0)} (|f| -6) + \sum_{v \in V(G)} (2d(v) -6) = -12.\]
\end{Proposition}

\subsection{Structural properties of a minimum counterexample}\label{propmincou}
In this section, we collect some properties of a minimum counterexample to Theorem \ref{girth9}.

Let $G$ be a counterexample to Theorem \ref{girth9} with $v(G) + e(G)$ minimized. We fix an embedding of $G$ and say $F$ is a face of $G$ or $F \in F(G)$ if it is a face of that embedding. We use the notation that a vertex $v \in V(G)$ is a {\it $k$-vertex} if it has degree $k$ in $G$.

\begin{Claim}\label{no1}
The graph $G$ contains no vertex of degree $1$.
\end{Claim}

\begin{proof}
Suppose otherwise, so there is some $v \in V(G)$ with $d_G(v)=1$. Let $u$ be the neighbour of $v$ in $G$ and let $G'=G-v$. As $G'$ is smaller than $G$, there is a $(\infty,1)$-bounded linear forest decomposition $(F_\infty,F_1)$ of $G'$. If $d_{F_\infty}(u)\leq 1$, let $F'_\infty$ be defined by $E(F'_\infty)=E(F_\infty)\cup uv$ and let $F'_1=F_1$. If $d_{F_\infty}(u)=2$, let $F'_1$ be defined by $E(F'_1)=E(F_1)\cup uv$ and $F'_\infty=F_\infty$. In either case, it follows that $(F'_\infty,F'_1)$ is a $(\infty,1)$-bounded linear forest decomposition of $G$, a contradiction to $G$ being a counterexample.

\end{proof}

\begin{Lemma}\label{ben}
The graph $G$ does not contain two adjacent $2$-vertices.
\end{Lemma}

\begin{proof}
Suppose otherwise, so there is an edge $e=uv\in E(G)$ with $d_G(u)=d_G(v)=2$. Let $G'=G-uv$. As $G'$ is smaller than $G$, there is a $(\infty,1)$-bounded linear forest decomposition $(F_\infty,F_1)$ of $G'$. If $d_{F_\infty}(u)+d_{F_\infty}(v)\leq 1$, let $F'_\infty$ be defined by $E(F'_\infty)=E(F_\infty)\cup uv$ and let $F'_1=F_1$. If $d_{F_\infty}(u)+d_{F_\infty}(v)=2$, let $F'_1$ be defined by $E(F'_1)=E(F_1)\cup uv$ and $F'_\infty=F_\infty$. In either case, it follows that $(F'_\infty,F'_1)$ is a $(\infty,1)$-bounded linear forest decomposition of $G$, a contradiction to $G$ being a counterexample.

\end{proof}

\begin{Lemma}\label{10e1}
Let $F$ be a non-degenerate face of size 10 of $G$ and $v_1,\ldots,v_{10}$ be a canonical ordering of $V(F)$ such that $v_i$ is a 3-vertex for $i \in \{1,3,5,7,9\}$ and $v_i$ is a 2-vertex for $i \in \{2,4,6,8,10\}$. Further, for $i \in \{1,3,5,7,9\}$, let $u_i$ be the unique neighbour of $v_i$ in $V(G)-V(F)$. Then for any $(\infty,1)$-bounded linear forest decomposition $(F_{\infty},F_1 )$ of $G-E(F)$, we have $\{u_1v_1,u_3v_3,u_5v_5,u_7v_7,u_9,v_9\} \\ \subseteq E(F_1)$. 
\end{Lemma}
\begin{proof}
Suppose otherwise. Let $(F_\infty',F_1')$ be defined by $E(F'_1)=E(F_1)\cup \{v_{i}v_{i+1}:i \in \{1,3,5,7,9\} \text{and } u_iv_i \in E(F_{\infty})\}$ and $E(F'_{\infty})=E(F_{\infty})\cup (E(F)-E(F'_1))$. It follows immediately from the construction that $F'_1$ is a matching. Further, every component of $F'_{\infty}$ is obtained by extending by a path of $F_{\infty}$. It follows that $F'_{\infty}$ is a linear forest, so $(F'_{\infty},F'_1)$ is a $(\infty,1)$-bounded linear forest decomposition of $G$, a contradiction. For an illustration, see Figure \ref{bild1}. 
\end{proof}

\begin{figure}[h!]
    \centering
        \includegraphics[width=.25\textwidth]{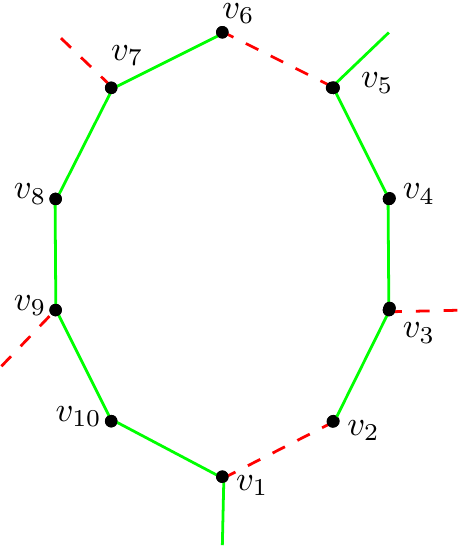}
        \caption{An illustration of the proof of Lemma \ref{10e1} where $u_3v_3,u_7v_7$ and $u_9v_9$ are included in $E(F_1)$. The dashed red edges are in $E(F'_1)$ and the solid green edges are in $E(F'_\infty)$.}\label{bild1}
\end{figure}

\begin{Lemma}\label{10dist4}
Let $F$ be a non-degenerate face of size 10 of $G$ and $v_1,\ldots,v_{10}$ be a canonical ordering of $V(F)$ such that $v_i$ is a 3-vertex for $i \in \{1,3,5,7,9\}$ and $v_i$ is a 2-vertex for $i \in \{2,4,6,8,10\}$. Then $dist_{G-E(F)}(v_1,v_5)\leq 6$.
\end{Lemma}
\begin{proof}
Suppose otherwise and let $H$ be obtained from $G-E(F)$ by adding a new vertex $z$ and the 3 new edges $v_1z,v_3z$ and $v_5z$.
\begin{Claim}
The girth of $H$ is at least 9.
\end{Claim}
\begin{proof}
Suppose otherwise, so $H$ contains a cycle $C$ of length at most 8. As $G-E(F)$ is of girth at least 9, we obtain that $z \in V(C)$ and $E(C)$ contains exactly 2 of the 3 edges $v_1z,v_3z$ and $v_5z$. If $v_3z \in E(C)$, say $\{v_1z,v_3z\}\subseteq E(C)$, then the graph $C'$ which is defined by $V(C')=V(C)-z\cup v_2$ and $E(C')=E(C)-\{v_1z,v_3z\}\cup \{v_1v_2,v_2v_3\}$ is a cycle of the same length as $C$ in $G$, a contradiction to the girth of $G$ being at least 9. If $\{v_1z,v_5z\}\subseteq V(C)$, then $C-z$ is a $v_1v_5$-path of length at most 6 in $G-E(F)$, a contradiction to the assumption.
\end{proof}
For $i \in \{1,3,5,7,9\}$, let $u_i$ be the unique neighbor of $v_i$ in $V(G)-V(F)$. As $H$ is clearly planar, subcubic and smaller than $G$, we obtain that $H$ has a $(\infty,1)$-bounded linear forest decomposition $(F_\infty,F_1)$. Further, if $\{u_1v_1,u_3v_3,u_5v_5\}\subseteq E(F_1)$, as $F_1$ is a matching, we obtain that $\{v_1z,v_3z,v_5z\}\subseteq E(F_{\infty})$, a contradiction to $E(F_{\infty})$ being a linear forest. Hence $(F_{\infty}-z,F_1-z)$ is a $(\infty,1)$-bounded linear forest decomposition of $G-E(F)$ that satisfies $\{u_1v_1,u_3v_3,u_5v_5,u_7v_7,u_9,v_9\}- E(F_1) \neq \emptyset$, a contradiction to Lemma \ref{10e1}.
\end{proof}

\begin{Lemma}\label{1063}
Every $10$-face of $G$ is incident to at least six 3-vertices.
\end{Lemma}
\begin{proof}
Let $F$ be a face of size 10 of $G$ and suppose for the sake of a contradiction that $V(F)$ contains at most five 3-vertices. Observe that $F$ is non-degenerate as the girth of $G$ is at least 9. Further, by Lemma \ref{ben}, we obtain that there is a canonical ordering $v_1,\ldots,v_{10}$ of $V(F)$ such that $v_i$ is a 3-vertex for $i \in \{1,3,5,7,9\}$ and $v_i$ is a 2-vertex for $i \in \{2,4,6,8,10\}$. 
 By Lemma \ref{10dist4}, we obtain that $dist_{G-E(F)}(v_1,v_5)\leq 6$. By applying Lemma \ref{10dist4} to the canonical ordering $v_3,\ldots,v_{10},v_1,v_2$, we obtain that $dist_{G-E(F)}(v_3,v_7)\leq 6$. Now Proposition \ref{wege} yields that there is some $i \in \{1,5\}$ such that $dist_{G-E(F)}(v_i,v_{i+2})\leq 6$, so there is a path $v_iv_{i+2}$-path $P$ of length at most 6 in $G-E(F)$. Now let $C$ be the subgraph of $G$ with $V(C)=V(P)\cup v_{i+1}$ and $E(C)=E(P)\cup \{v_iv_{i+1},v_{i+1}v_{i+2}\}$. Then $C$ is a cycle of length at most 8 in $G$, a contradiction.
\end{proof}

\begin{Lemma}\label{9cas}
Let $F$ be a non-degenerate 9-face of $G$ and $v_1,\ldots,v_9$ a canonical ordering of $V(F)$ such that $v_i$ is a 3-vertex for $i\in \{1,2,4,6,8\}$ and $v_i$ is a 2-vertex for $i\in \{3,5,7,9\}$.  Further, for $i \in \{1,2,4,6,8\}$, let $u_i$ be the unique neighbour of $v_i$ in $V(G)-V(F)$. Then for any $(\infty,1)$-bounded linear forest decomposition $(F_\infty,F_1)$ of $G-E(F)$, one of the following holds:

\begin{enumerate}
\item $\{u_1v_1,u_2v_2,u_4v_4,u_6v_6,u_8v_8\}\subseteq E(F_1)$,
\item $\{u_4v_4,u_6v_6,u_8v_8\}\subseteq E(F_1)$ and $\{u_1v_1,u_2v_2\}\subseteq E(P)$ for some component $P$ of $F_\infty$.
\end{enumerate}
\end{Lemma}
\begin{proof}
Let $(F_\infty,F_1)$ be a $(\infty,1)$-bounded linear forest decomposition of $G-E(F)$. We distinguish several cases.

\begin{case}\label{e10}
$\{u_1v_1,u_2v_2,u_4v_4,u_6v_6,u_8v_8\}\cap  E(F_1) = \emptyset$.
\end{case}
\begin{proof}
For $i \in \{1,2,4,6,8\}$, let $P_i$ be the connected component of $F_\infty$ that contains the edge $u_iv_i$. By Proposition \ref{3paths}, we obtain that one of $P_4 \neq P_6$ and  $P_6 \neq P_8$ holds, By symmetry, we may suppose that $P_4 \neq P_6$. Let $(F'_\infty,F_1)$ be defined by $E(F'_1)=E(F_1)\cup \{v_1v_2,v_3v_4,v_6v_7,v_8v_9\}$ and  $E(F'_\infty)=E(F_\infty)\cup \{v_2v_3,v_4v_5,v_5v_6,v_7v_8,v_9v_1\}$. It follows by construction that $F_1'$ is a matching, and since $P_4 \neq P_6$, that $F'_\infty$ is a linear forest. Hence $(F'_\infty,F'_1)$ is a $(\infty,1)$-bounded linear forest decomposition of $G$, a contradiction. For an illustration, see Figure \ref{bild2}.
\end{proof}

\begin{figure}[h!]
    \centering
        \includegraphics[width=.25\textwidth]{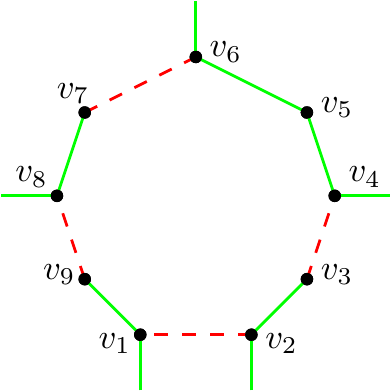}
        \caption{An illustration of the proof of Case \ref{e10}. The dashed red edges are in $E(F'_1)$ and the solid green edges are in $E(F'_\infty)$.}\label{bild2}
\end{figure}
\begin{case}
$|\{u_1v_1,u_2v_2,u_4v_4,u_6v_6,u_8v_8\}\cap  E(F_1)| = 1$.
\end{case}
\begin{proof}
By symmetry, we may suppose that the unique edge in $\{u_1v_1,u_2v_2,u_4v_4,u_6v_6,u_8v_8\}\cap  E(F_1)$ is one of $u_2v_2,u_4v_4$ and $u_6v_6$.
\begin{Subcase}\label{e111}
$\{u_1v_1,u_2v_2,u_4v_4,u_6v_6,u_8v_8\}\cap  E(F_1) = \{u_2v_2\}$.
\end{Subcase}
\begin{proof}
Let $(F'_\infty,F_1')$ be defined by $E(F'_1)=E(F_1)\cup \{v_3v_4,v_5v_6,v_7v_8,v_9v_1\}$ and  $E(F'_\infty)=E(F_\infty)\cup \{v_1v_2,v_2v_3,v_4v_5,v_6v_7,v_8v_9\}$. It follows by construction that $E(F_1')$ is a matching and that $E(F'_\infty)$ is a linear forest. Hence $(F'_\infty,F'_1)$ is a $(\infty,1)$-bounded linear forest decomposition of $G$, a contradiction. For an illustration, see Figure \ref{bild3}.
\end{proof}

\begin{figure}[h!]
    \centering
        \includegraphics[width=.25\textwidth]{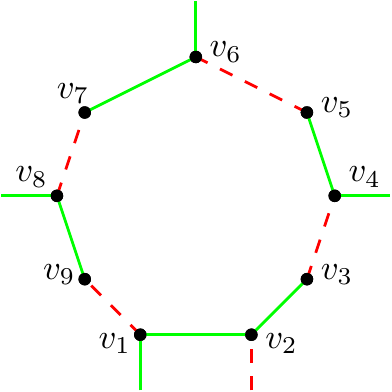}
        \caption{An illustration of the proof of Subcase \ref{e111}. The dashed red edges are in $E(F'_1)$ and the solid green edges are in $E(F'_\infty)$.}\label{bild3}
\end{figure}

\begin{Subcase}\label{e112}
$\{u_1v_1,u_2v_2,u_4v_4,u_6v_6,u_8v_8\}\cap  E(F_1) = \{u_4v_4\}$.
\end{Subcase}
\begin{proof}
For $i \in \{1,2,6,8\}$, let $P_i$ be the connected component of $E(F_\infty)$ that contains the edge $u_iv_i$. By Proposition \ref{3paths}, we obtain that one of $P_1 \neq P_2$ and  $P_2 \neq P_6$ holds.

First, suppose that $P_1 \neq P_2$ holds. Let $(F'_\infty,F'_1)$ be the decomposition of $G$ defined by $E(F'_1)=E(F_1)\cup \{v_2v_3,v_5v_6,v_7v_8,v_9v_1\}$ and  $E(F'_\infty)=E(F_\infty)\cup \{v_1v_2,v_3v_4,v_4v_5,v_6v_7,v_8v_9\}$. It follows by construction that $F_1'$ is a matching and since $P_1 \neq P_2$ that $F'_\infty$ is a linear forest. Hence $(F'_\infty,F'_1)$ is a $(\infty,1)$-bounded linear forest decomposition of $G$, a contradiction to $G$ being a counterexample. For an illustration, see Figure \ref{bild4}.

\begin{figure}[h!]
    \centering
        \includegraphics[width=.25\textwidth]{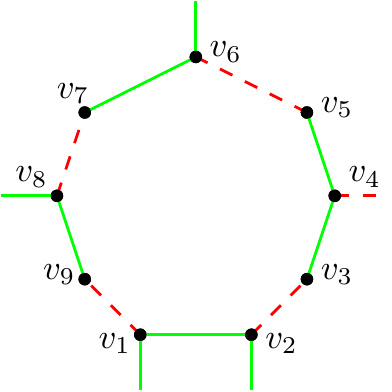}
        \caption{An illustration of the first part of the proof of Subcase \ref{e112}. The dashed red edges are in $E(F'_1)$ and the solid green edges are in $E(F'_\infty)$.}\label{bild4}
\end{figure}

Now suppose that $P_2 \neq P_6$ holds. Let $(F'_\infty,F'_1)$ be the decomposition of $G$ defined by $E(F'_1)=E(F_1)\cup \{v_1v_2,v_6v_7,v_8v_9\}$ and  $E(F'_\infty)=E(F_\infty)\cup \{v_2v_3,v_3v_4,v_4v_5,v_5v_6,v_7v_8,v_9v_1\}$. It follows by construction that $F_1'$ is a matching and since $P_2 \neq P_6$ that $F'_\infty$ is a linear forest. Hence $(F'_\infty,F'_1)$ is a $(\infty,1)$-bounded linear forest decomposition of $G$, a contradiction to $G$ being a counterexample. For an illustration, see Figure \ref{bild5}.

\end{proof}

\begin{figure}[h!]
    \centering
        \includegraphics[width=.25\textwidth]{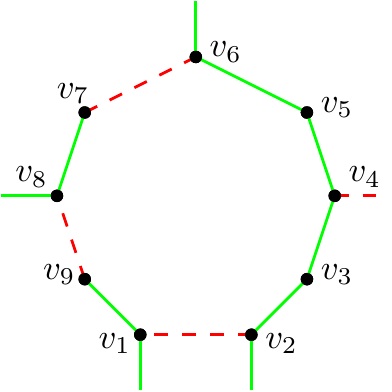}
        \caption{An illustration of the  second part of the proof of Subcase \ref{e112}. The dashed red edges are in $E(F'_1)$ and the solid green edges are in $E(F'_\infty)$.}\label{bild5}
\end{figure}
\begin{Subcase}\label{e113}
$\{u_1v_1,u_2v_2,u_4v_4,u_6v_6,u_8v_8\}\cap  E(F_1) = \{u_6v_6\}$.
\end{Subcase}
\begin{proof}
For $i \in \{1,2,4,8\}$, let $P_i$ be the connected component of $F_\infty$ that contains the edge $u_iv_i$. By Proposition \ref{3paths}, we obtain that one of $P_1 \neq P_2$ and  $P_2 \neq P_4$ holds.

First, suppose that $P_1 \neq P_2$ holds. Let $(F'_\infty,F'_1)$ be the decomposition of $G$ defined by $E(F'_1)=E(F_1)\cup \{v_2v_3,v_4v_5,v_7v_8,v_9v_1\}$ and  $E(F'_\infty)=E(F_\infty)\cup \{v_1v_2,v_3v_4,v_5v_6,v_6v_7,v_8v_9\}$. It follows by construction that $E(F_1')$ is a matching and since  $P_1 \neq P_2$ that $F'_\infty$ is a linear forest. Hence $(F'_\infty,F'_1)$ is a $(\infty,1)$-bounded linear forest decomposition of $G$, a contradiction to $G$ being a counterexample. For an illustration, see Figure \ref{bild6}.

\begin{figure}[h!]
    \centering
        \includegraphics[width=.25\textwidth]{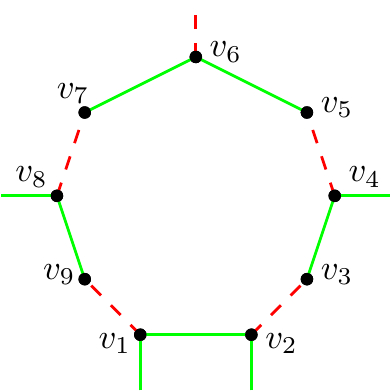}
        \caption{An illustration of the first part of the proof of Subcase \ref{e113}. The dashed red edges are in $E(F'_1)$ and the solid green edges are in $E(F'_\infty)$.}\label{bild6}
\end{figure}

Now suppose that $P_2 \neq P_4$ holds. Let $(F'_\infty,F'_1)$ be the decomposition of $G$ defined by $E(F'_1)=E(F_1)\cup \{v_1v_2,v_4v_5,v_8v_9\}$ and  $E(F'_\infty)=E(F_\infty)\cup \{v_2v_3,v_3v_4,v_5v_6,v_6v_7, v_7v_8,v_9v_1\}$. It follows by construction that $F_1'$ is a matching and since $P_2 \neq P_4$ that $F'_\infty$ is a linear forest. Hence $(F'_\infty,F'_1)$ is a $(\infty,1)$-bounded linear forest decomposition of $G$, a contradiction to $G$ being a counterexample. For an illustration, see Figure \ref{bild7}.
\end{proof}
\begin{figure}[h!]
    \centering
        \includegraphics[width=.25\textwidth]{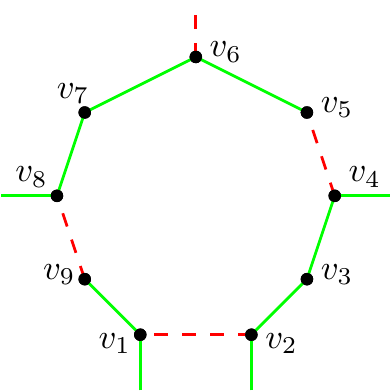}
        \caption{An illustration of the second part of the proof of Subcase \ref{e113}. The dashed red edges are in $E(F'_1)$ and the solid green edges are in $E(F'_\infty)$.}\label{bild7}
\end{figure}
\end{proof}
\begin{case}
$|\{u_1v_1,u_2v_2,u_4v_4,u_6v_6,u_8v_8\}\cap  E(F_1)| = 2$.
\end{case}
\begin{proof}
By symmetry, we may suppose that one of the following occurs:
\begin{itemize}
\item $\{u_1v_1,u_2v_2,u_4v_4,u_6v_6,u_8v_8\}\cap  E(F_1)=\{u_1v_1,u_2v_2\}$,
\item $\{u_1v_1,u_2v_2,u_4v_4,u_6v_6,u_8v_8\}\cap  E(F_1)=\{u_2v_2,u_4v_4\}$,
\item $\{u_1v_1,u_2v_2,u_4v_4,u_6v_6,u_8v_8\}\cap  E(F_1)=\{u_2v_2,u_6v_6\}$,
\item $\{u_1v_1,u_2v_2,u_4v_4,u_6v_6,u_8v_8\}\cap  E(F_1)=\{u_4v_4,u_6v_6\}$,
\item $\{u_1v_1,u_2v_2,u_4v_4,u_6v_6,u_8v_8\}\cap  E(F_1)=\{u_4v_4,u_8v_8\}$.
\end{itemize}
\begin{Subcase}\label{e121}
$\{u_1v_1,u_2v_2,u_4v_4,u_6v_6,u_8v_8\}\cap  E(F_1)=\{u_1v_1,u_2v_2\}$.
\end{Subcase}
\begin{proof}
For $i \in \{4,6,8\}$, let $P_i$ be the connected component of $F_\infty$ that contains the edge $u_iv_i$. By Proposition \ref{3paths}, we obtain that one of $P_4 \neq P_6$ and  $P_6 \neq P_8$ holds, By symmetry, we may suppose that $P_4 \neq P_6$. Let $(F'_\infty,F'_1)$ be the decomposition of $G$ defined by $E(F'_1)=E(F_1)\cup \{v_3v_4,v_6v_7,v_8v_9\}$ and  $E(F'_\infty)=E(F_\infty)\cup \{v_1v_2,v_2v_3,v_4v_5,v_5v_6,v_7v_8,v_9v_1\}$. It follows by construction that $F_1'$ is a matching and since  $P_4 \neq P_6$ that $F'_\infty$ is a linear forest. Hence $(F'_\infty,F'_1)$ is a $(\infty,1)$-bounded linear forest decomposition of $G$, a contradiction to $G$ being a counterexample. For an illustration, see Figure \ref{bild8}.
\end{proof}
\begin{figure}[h!]
    \centering
        \includegraphics[width=.25\textwidth]{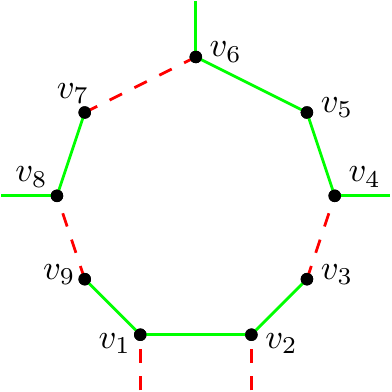}
        \caption{An illustration of the proof of Subcase \ref{e121}. The dashed red edges are in $E(F'_1)$ and the solid green edges are in $E(F'_\infty)$.}\label{bild8}
\end{figure}
\begin{Subcase}\label{e122}
$\{u_1v_1,u_2v_2,u_4v_4,u_6v_6,u_8v_8\}\cap  E(F_1)=\{u_2v_2,u_4v_4\}$.
\end{Subcase}
\begin{proof}
Let $(F'_\infty,F'_1)$ be the decomposition of $G$ defined by $E(F'_1)=E(F_1)\cup \{v_5v_6,v_7v_8,v_9v_1\}$ and  $E(F'_\infty)=E(F_\infty)\cup \{v_1v_2,v_2v_3,v_3v_4,v_4v_5,v_6v_7,v_8v_9\}$. It follows by construction that $F_1'$ is a matching and that $F'_\infty$ is a linear forest. Hence $(F'_\infty,F'_1)$ is a $(\infty,1)$-bounded linear forest decomposition of $G$, a contradiction to $G$ being a counterexample. For an illustration, see Figure \ref{bild9}.
\end{proof}
\begin{figure}[h!]
    \centering
        \includegraphics[width=.25\textwidth]{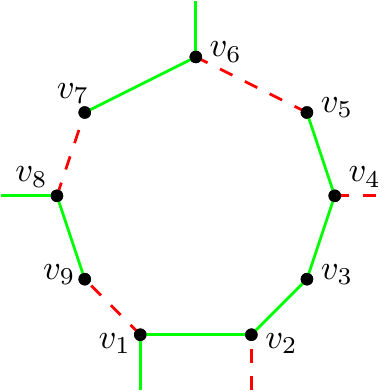}
        \caption{An illustration of the proof of Subcase \ref{e122}. The dashed red edges are in $E(F'_1)$ and the solid green edges are in $E(F'_\infty)$.}\label{bild9}
\end{figure}
\begin{Subcase}\label{e123}
$\{u_1v_1,u_2v_2,u_4v_4,u_6v_6,u_8v_8\}\cap  E(F_1)=\{u_2v_2,u_6v_6\}$.
\end{Subcase}
\begin{proof}
 Let $(F'_\infty,F'_1)$ be the decomposition of $G$ defined by $E(F'_1)=E(F_1)\cup \{v_3v_4,v_7v_8,v_9v_1\}$ and  $E(F'_\infty)=E(F_\infty)\cup \{v_1v_2,v_2v_3,v_4v_5,v_5v_6,v_6v_7,v_8v_9\}$. It follows by construction that $F_1'$ is a matching and that $F'_\infty$ is a linear forest. Hence $(F'_\infty,F'_1)$ is a $(\infty,1)$-bounded linear forest decomposition of $G$, a contradiction to $G$ being a counterexample. For an illustration, see Figure \ref{bild10}.
\end{proof}
\begin{figure}[h!]
    \centering
        \includegraphics[width=.25\textwidth]{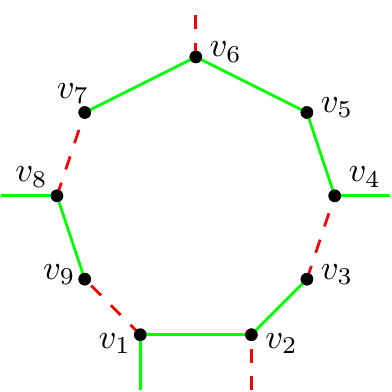}
        \caption{An illustration of the proof of Subcase \ref{e123}. The dashed red edges are in $E(F'_1)$ and the solid green edges are in $E(F'_\infty)$.}\label{bild10}
\end{figure}
\begin{Subcase}\label{e124}
$\{u_1v_1,u_2v_2,u_4v_4,u_6v_6,u_8v_8\}\cap  E(F_1)=\{u_4v_4,u_6v_6\}$.
\end{Subcase}
\begin{proof}
 For $i \in \{1,2,8\}$, let $P_i$ be the connected component of $F_\infty$ that contains the edge $u_iv_i$. By Proposition \ref{3paths}, we obtain that one of $P_1\neq P_2$ and  $P_2 \neq P_8$ holds.

First, suppose that $P_1 \neq P_2$ holds.  Let $(F'_\infty,F'_1)$ be the decomposition of $G$ defined by $E(F'_1)=E(F_1)\cup \{v_2v_3,v_7v_8,v_9v_1\}$ and  $E(F'_\infty)=E(F_\infty)\cup \{v_1v_2,v_3v_4,v_4v_5,v_5v_6, v_6v_7,v_8v_9\}$. It follows by construction that $F_1'$ is a matching and since  $P_1 \neq P_2$ that $F'_\infty$ is a linear forest. Hence $(F'_\infty,F'_1)$ is a $(\infty,1)$-bounded linear forest decomposition of $G$, a contradiction to $G$ being a counterexample. For an illustration, see Figure \ref{bild11}.
\begin{figure}[h!]
    \centering
        \includegraphics[width=.25\textwidth]{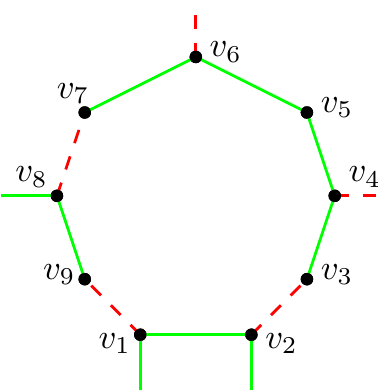}
        \caption{An illustration of the first part of the proof of Subcase \ref{e124}. The dashed red edges are in $E(F'_1)$ and the solid green edges are in $E(F'_\infty)$.}\label{bild11}
\end{figure}

Now suppose that $P_2\neq P_8$ holds. Let $(F'_\infty,F'_1)$ be the decomposition of $G$ defined by $E(F'_1)=E(F_1)\cup \{v_1v_2,v_8v_9\}$ and  $E(F'_\infty)=E(F_\infty)\cup \{v_2v_3,v_3v_4,v_4v_5, \\ v_5v_6, v_6v_7,v_7v_8,v_9v_1\}$. It follows by construction that $F_1'$ is a matching and since $P_2 \neq P_8$ that $F'_\infty$ is a linear forest. Hence $(F'_\infty,F'_1)$ is a $(\infty,1)$-bounded linear forest decomposition of $G$, a contradiction to $G$ being a counterexample. For an illustration, see Figure \ref{bild12}.
\end{proof}
\begin{figure}[h!]
    \centering
        \includegraphics[width=.25\textwidth]{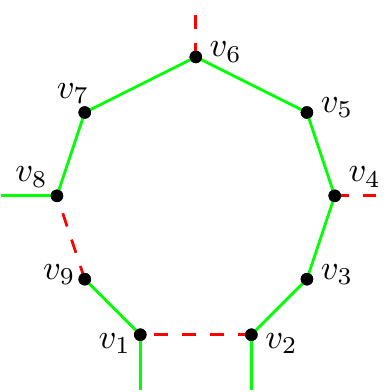}
        \caption{An illustration of the second part of the proof of Subcase \ref{e124}. The dashed red edges are in $E(F'_1)$ and the solid green edges are in $E(F'_\infty)$.}\label{bild12}
\end{figure}
\begin{Subcase}\label{e125}
\item $\{u_1v_1,u_2v_2,u_4v_4,u_6v_6,u_8v_8\}\cap  E(F_1)=\{u_4v_4,u_8v_8\}$.
\end{Subcase}
\begin{proof}
For $i \in \{1,2,6\}$, let $P_i$ be the connected component of $F_\infty$ that contains the edge $u_iv_i$. By Proposition \ref{3paths}, we obtain that one of $P_1 \neq P_6$ and  $P_2 \neq P_6$ holds, By symmetry, we may suppose that $P_2 \neq P_6$. Let $(F'_\infty,F'_1)$ be the decomposition of $G$ defined by $E(F'_1)=E(F_1)\cup \{v_1v_2,v_6v_7\}$ and  $E(F'_\infty)=E(F_\infty)\cup \{v_2v_3,v_3v_4,v_4v_5,v_5v_6,v_7v_8,v_8v_9,v_9v_1\}$. It follows by construction that $F_1'$ is a matching and since  $P_2 \neq P_6$ that $F'_\infty$ is a linear forest. Hence $(F'_\infty,F'_1)$ is a $(\infty,1)$-bounded linear forest decomposition of $G$, a contradiction to $G$ being a counterexample. For an illustration, see Figure \ref{bild13}.
\end{proof}
\begin{figure}[h!]
    \centering
        \includegraphics[width=.25\textwidth]{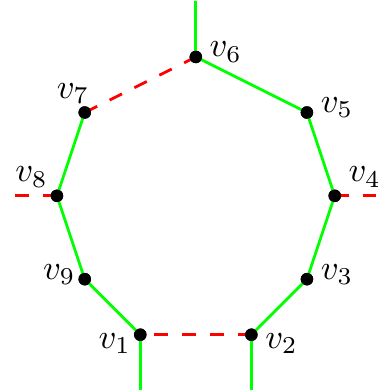}
        \caption{An illustration of the proof of Subcase \ref{e125}. The dashed red edges are in $E(F'_1)$ and the solid green edges are in $E(F'_\infty)$.}\label{bild13}
\end{figure}
\end{proof}
\begin{case}
$|\{u_1v_1,u_2v_2,u_4v_4,u_6v_6,u_8v_8\}\cap  E(F_1)| = 3$.
\end{case}
\begin{proof}
We need to distinguish two subcases.
\begin{Subcase}\label{e131}
\item $\{u_1v_1,u_2v_2,u_4v_4,u_6v_6,u_8v_8\}\cap  E(F_\infty)=\{u_1v_1,u_2v_2\}$.
\end{Subcase}
\begin{proof}
For $i \in \{1,2\}$, let $P_i$ be the connected component of $F_\infty$ that contains the edge $u_iv_i$. If $P_1=P_2$, then $(ii)$ holds, so there is nothing to prove. We may hence suppose that $P_1\neq P_2$. Let $(F'_\infty,F'_1)$ be the decomposition of $G$ defined by $E(F'_1)=E(F_1)\cup v_1v_2$ and  $E(F'_\infty)=E(F_\infty)\cup \{v_2v_3,v_3v_4,v_4v_5,v_5v_6,v_6v_7,v_7v_8,v_8v_9,v_9v_1\}$. It follows by construction that $F_1'$ is a matching and since $P_1 \neq P_2$ that $F'_\infty$ is a linear forest. Hence $(F'_\infty,F'_1)$ is a $(\infty,1)$-bounded linear forest decomposition of $G$, a contradiction to $G$ being a counterexample. For an illustration, see Figure \ref{bild14}.
\end{proof}
\begin{figure}[h!]
    \centering
        \includegraphics[width=.25\textwidth]{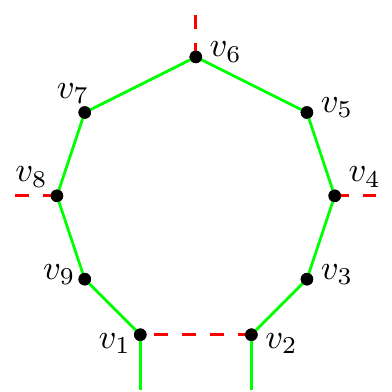}
        \caption{An illustration of the proof of Subcase \ref{e131}. The dashed red edges are in $E(F'_1)$ and the solid green edges are in $E(F'_\infty)$.}\label{bild14}
\end{figure}
\begin{Subcase}\label{e132}
\item $\{u_1v_1,u_2v_2,u_4v_4,u_6v_6,u_8v_8\}\cap  E(F_\infty)\neq\{u_1v_1,u_2v_2\}$.
\end{Subcase}
\begin{proof}
By symmetry, we may suppose that $u_1v_1\in E(F_1)$.  Let $(F'_\infty,F'_1)$ be the decomposition of $G$ defined by $E(F'_1)=E(F_1)\cup \{v_iv_{i+1}:i \in \{2,4,6,8\},u_iv_i \in E(F_\infty)\}$ and $E(F'_\infty)=E(F_\infty) \cup(E(F)-E(F'_1))$.  It follows by construction that $F_1'$ is a matching and that $F'_\infty$ is a linear forest. Hence $(F'_\infty,F'_1)$ is a $(\infty,1)$-bounded linear forest decomposition of $G$, a contradiction to $G$ being a counterexample. For an illustration, see Figure \ref{bild15}.
\end{proof}
\begin{figure}[h!]
    \centering
        \includegraphics[width=.25\textwidth]{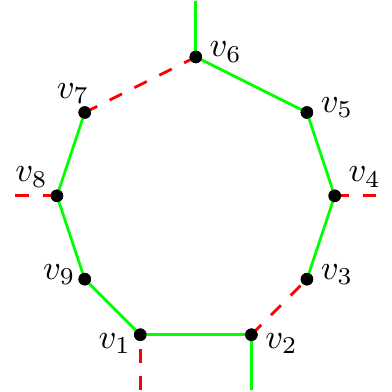}
        \caption{An illustration of the proof of Subcase \ref{e132} where $u_2v
_2$ and $u_6v_6$ are included in $E(F_\infty)$. The dashed red edges are in $E(F'_1)$ and the solid green edges are in $E(F'_\infty)$.}\label{bild15}
\end{figure}
\end{proof}
\begin{case}\label{e14}
$|\{u_1v_1,u_2v_2,u_4v_4,u_6v_6,u_8v_8\}\cap  E(F_1)| \geq 4$.
\end{case}
\begin{proof}
If $|\{u_1v_1,u_2v_2,u_4v_4,u_6v_6,u_8v_8\}\cap  E(F_1)| = 5$, then $(i)$ holds, so there is nothing to prove. We may hence suppose that $|\{u_1v_1,u_2v_2,u_4v_4,u_6v_6,u_8v_8\}\cap  E(F_1)| = 4$. By symmetry, we may suppose that $u_1v_1\in E(F_1)$. Let $i \in \{2,4,6,8\}$ be the unique integer such that $u_iv_i\in E(F_\infty)$. Let $(F'_\infty,F'_1)$ be the decomposition of $G$ defined by $E(F'_1)=E(F_1)\cup v_iv_{i+1}$ and $E(F'_\infty)=E(F_\infty) \cup(E(F)-v_iv_{i+1})$. It follows by construction that $F_1'$ is a matching and that $F'_\infty$ is a linear forest. Hence $(F'_\infty,F'_1)$ is a $(\infty,1)$-bounded linear forest decomposition of $G$, a contradiction to $G$ being a counterexample. For an illustration, see Figure \ref{bild16}.
\end{proof}
\begin{figure}[h!]
    \centering
        \includegraphics[width=.25\textwidth]{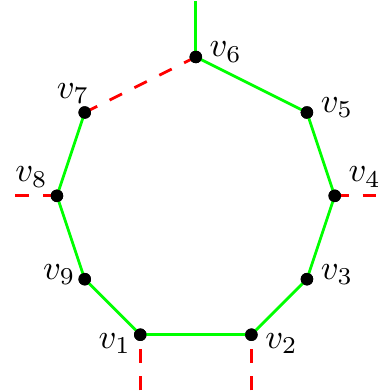}
        \caption{An illustration of the proof of Case \ref{e14} where $u_6v_6$ is included in $E(F_\infty)$. The dashed red edges are in $E(F'_1)$ and the solid green edges are in $E(F'_\infty)$.}\label{bild16}
\end{figure}
Now the case distinction is complete and so the proof of Lemma \ref{9cas} is finished.
\end{proof}

\begin{Lemma}\label{468}
Let $F$ be a non-degenerate 9-face of $G$ and $v_1,\ldots,v_9$ a canonical ordering of $V(F)$ such that $v_i$ is a 3-vertex for $i\in \{1,2,4,6,8\}$ and $v_i$ is a 2-vertex for $i\in \{3,5,7,9\}$. Then $dist_{G-E(F)}(v_4,v_8)\leq 6$.
\end{Lemma}
\begin{proof}
Suppose otherwise and let $H$ be obtained from $G-E(F)$ by adding a new vertex $z$ and the 3 new edges $v_4z,v_6z$ and $v_8z$.
\begin{Claim}
The girth of $H$ is at least 9.
\end{Claim}
\begin{proof}
Suppose otherwise, so $G$ contains a cycle $C$ of length at most 8. As $G-E(F)$ is of girth at least 9, we obtain that $z \in V(C)$ and $E(C)$ contains exactly 2 of the 3 edges $v_4z,v_6z$ and $v_8z$. If $v_6z \in E(C)$, say $\{v_4z,v_6z\}\subseteq E(C)$, then the graph $C'$ which is defined by $V(C')=V(C)-z\cup v_5$ and $E(C')=E(C)-\{v_4z,v_6z\}\cup \{v_4v_5,v_5v_6\}$ is a cycle of the same length as $C$ in $G$, a contradiction to the girth of $G$ being at least 9. If $\{v_4z,v_8z\}\subseteq V(C)$, then $C-z$ is a $v_4v_8$-path of length at most 6 in $G-E(F)$, a contradiction to the assumption.
\end{proof}
For $i \in \{1,2,4,6,8\}$, let $u_i$ be the unique neighbour of $v_i$ in $V(G)-V(F)$. As $H$ is clearly planar, subcubic and smaller than $G$, we obtain that $H$ has a $(\infty,1)$-bounded linear forest decomposition $(F_\infty,F_1)$. Further, if $\{u_4v_4,u_6v_6,u_8v_8\}\subseteq E(F_1)$, as $F_1$ is a matching, we obtain that $\{v_4z,v_6z,v_8z\}\subseteq E(F_{\infty})$, a contradiction to $F_{\infty}$ being a linear forest. Hence $(F_{\infty}-z,F_1-z)$ is a $(\infty,1)$-bounded linear forest decomposition of $G-E(F)$ that satisfies none of $(i)$ and $(ii)$, a contradiction to Lemma \ref{10e1}.
\end{proof}

\begin{Lemma}\label{246}
Let $F$ be a non-degenerate 9-face of $G$ and $v_1,\ldots,v_9$ a canonical ordering of $V(F)$ such that $v_i$ is a 3-vertex for $i\in \{1,2,4,6,8\}$ and $v_i$ is a 2-vertex for $i\in \{3,5,7,9\}$. Then $\min\{dist_{G-E(F)}(v_1,v_6),dist_{G-E(F)}(v_2,v_6)\}\leq 6$.
\end{Lemma}
\begin{proof}
Suppose otherwise and let $H$ be obtained from $G-(E(F)-v_1v_2)$ by adding a new vertex $z$ and the 3 new edges $v_2z,v_4z$ and $v_6z$.
\begin{Claim}
The girth of $H$ is at least 9.
\end{Claim}
\begin{proof}
Suppose otherwise, so $G$ contains a cycle $C$ of length at most 8. As $G-(E(F)-v_1v_2)$ is of girth at least 9, we obtain that $z \in V(C)$ and $E(C)$ contains exactly 2 of the 3 edges $v_2z,v_4z$ and $v_6z$. If $\{v_2z,v_4z\}\subseteq E(C)$, then the graph $C'$ which is defined by $V(C')=V(C)-z\cup v_3$ and $E(C')=E(C)-\{v_2z,v_4z\}\cup \{v_2v_3,v_3v_4\}$ is a cycle of the same length as $C$ in $G$, a contradiction to the girth of $G$ being at least 9. If $\{v_4z,v_6z\}\subseteq E(C)$, we similarly obtain a contradiction. If $\{v_2z,v_6z\}\subseteq V(C)$, then $C-z$ is a $v_2v_6$-path of length at most 6 in $G-(E(F)-v_1v_2)$. If $v_1v_2 \notin E(C)$,  then $C-z$ is a $v_2v_6$-path of length at most 6 in $G-(E(F))$, a contradiction to the assumption. If $v_1v_2 \in E(C)$, then $C-\{z,v_2\}$ is a $v_1v_6$-path of length at most 5 in $G-(E(F))$, a contradiction to the assumption.
\end{proof}
For $i \in \{1,2,4,6,8\}$, let $u_i$ be the unique neighbour of $v_i$ in $V(G)-V(F)$. As $H$ is clearly planar, subcubic and smaller than $G$, we obtain that $H$ has a $(\infty,1)$-bounded linear forest decomposition $(F_\infty,F_1)$. Observe that $(F'_\infty=F_{\infty}-z,F_1'=F_1-z)$ is a $(\infty,1)$-bounded linear forest decomposition of $G-E(F)$. If $(F'_\infty,F'_1)$ satisfies $(i)$, then we have in particular $\{u_2v_2,u_4v_4,u_6v_6\}\subseteq E(F'_1)$. As $F_1$ is a matching, this yields $\{v_2z,v_4z,v_6z\}\subseteq E(F_\infty)$, a contradiction to $F_\infty$ being a linear forest. If $(F'_\infty,F'_1)$ satisfies $(ii)$, then we obtain $\{u_4v_4,u_6v_6\}\subseteq E(F'_1)$. Further, as $F_\infty$ is a linear forest, we have $v_1v_2\in E(F_1)$. As $F_1$ is a matching, we obtain again $\{v_2z,v_4z,v_6z\}\subseteq E(F_\infty)$, a contradiction to $F_\infty$ being a linear forest.  Hence $(F'_\infty,F'_1)$ satisfies none of $(i)$ and $(ii)$, a contradiction to Lemma \ref{9cas}.
\end{proof}

\begin{Lemma}\label{963}
Every 9-face of $G$ is incident to at least six 3-vertices.
\end{Lemma}
\begin{proof}
Let $F$ be a face of size 9 of $G$ and suppose for the sake of a contradiction that $V(F)$ contains at most 5 6-vertices. Observe that $F$ is non-degenerate as the girth of $G$ is at least 9. Further, by Lemma \ref{ben}, we obtain that there is a canonical ordering $v_1,\ldots,v_{9}$ of $V(F)$ such that $v_i$ is a 3-vertex for $i \in \{1,2,4,6,8\}$ and $v_i$ is a 2-vertex for $i \in \{3,5,7,9\}$. 
 By Lemma \ref{468}, we obtain that $dist_{G-E(F)}(v_4,v_8)\leq 6$. By Lemma \ref{246}, we obtain that $\min\{dist_{G-E(F)}(v_1,v_6),dist_{G-E(F)}(v_2,v_6)\}\leq 6$. By symmetry, we may suppose that $dist_{G-E(F)}(v_2,v_6)\leq 6$. Now Proposition \ref{wege} yields that there is some $i \in \{2,6\}$ such that $dist_{G-E(F)}(v_i,v_{i+2})\leq 6$, so there is a path $v_iv_{i+2}$-path $P$ of length at most 6 in $G-E(F)$. Now let $C$ be the subgraph of $G$ with $V(C)=V(P)\cup v_{i+1}$ and $E(C)=E(P)\cup \{v_iv_{i+1},v_{i+1}v_{i+2}\}$. Then $C$ is a cycle of length at most 8 in $G$, a contradiction.
\end{proof}

\subsection{Discharging}\label{disch}

We are now ready to prove Theorem \ref{girth9}.
\begin{proof}(of Theorem \ref{girth9})
Assign to every $F \in F(G)$ an initial charge of $|F|-6$ and assign to every $v \in V(G)$  an initial charge of $2d_G(v)-6$. Consider the following discharging rule:
Every face sends a charge of 1 to each 2-vertex that appears once on its boundary walk and a charge of 2 to every 2-vertex that appears twice on its boundary walk. Let $\text{ch}_{f}$ denote the final charge of the vertices and faces. We will show the final charge is non-negative.
\begin{Claim}\label{discharge}
The final charge of every face $x \in F(G)$ and every vertex $x \in V(G)$ is at least 0.
\end{Claim}
\begin{proof}
First, consider a nondegenerate face $F$ of $G$. If $|F|\geq 11$, then by Claims \ref{no1} and Lemma \ref{ben}, we obtain that the boundary walk of $F$ has at least half of its vertices being 3-vertices, and in particular, at least six $3$-vertices. If $|F|= 10$, we obtain by Lemma \ref{1063} that $V(F)$ contains at least six 3-vertices. If $|F|= 9$, we obtain by Lemma \ref{963} that $V(F)$ contains at least six 3-vertices.  Hence $F$ sends $1$ charge to at most $|F|-6$ vertices. and thus  $\text{ch}_{f}(F)\geq 0$ follows.

Now consider a degenerate face $F$ of $G$. Let $v_{0},v_{1},\ldots,v_{|F|-1}$ be a boundary walk of $F$. By Claim \ref{no1}, and Lemma \ref{ben}, it follows if $v_{i}$ is a $2$-vertex, then $v_{i+1}$ is a $3$-vertex in the boundary walk for all $i \in \{0,\ldots,|F|\}$, where indices are taken modulo $|F|$. Further observe that $|F|\geq 11$ as the girth of $G$ is 9. This yields that there is some $I\subseteq \{0,\ldots,|F|-1\}$ such that $|I|\geq 6$ and $v_i$ is a 3-vertex for all $i \in I$. Therefore $\text{ch}_{f}(F) \geq (|F|-6)-(|F|-|I|)\geq 0$, as desired.

Now consider a vertex $v \in V(G)$. By Claim \ref{no1} and as $G$ is subcubic, we have $d_G(v)\in \{2,3\}$. If $d_G(v)=3$, then $v$ does neither send nor receive any charge and so its final charge is $2d_G(v)-6=0$. If $d_G(v)=2$ and $v$ is incident to two distinct faces, then $v$ does not send any charge and receives a charge of 1 from both the faces it is incident to. It follows that its final charge is $c(v)=2d_G(v)-6+1+1=0$. If $d_G(v)=2$ and $v$ is incident to a single face, then $v$ does not send any charge and receives a charge of 2 from the face it is incident to. It follows that the final charge of $v$ is $\text{ch}_{f}(v)=2d_G(v)-6+2=0$.
\end{proof}
As the total charge remains unchanged throughout the discharging procedure, we obtain \[\sum_{F \in F(G)} (|F| -6) + \sum_{v \in V(G)} (2d_G(v) -6) = \sum_{F \in F(G)} \text{ch}_{f}(F) + \sum_{v \in V(G)} \text{ch}_{f}(v)\geq 0>-12,\] a contradiction to Proposition \ref{facechargingobs}. This finishes the proof.
\end{proof}
\section{Conclusion}
In this article, we dealt with the problem of decomposing a graph into two linear forests whose components have bounded length, considering both algorithmic and structural questions.

This work leaves many questions open. For the algorithmic part, the most obvious question is to close the dichotomy left open by Theorems \ref{21poly} and \ref{mainhard}.
\begin{Problem}
Determine the complexity of $(k,1)$-BLFD for $k \in \{3,\ldots,8\}$.
\end{Problem}
Further, for the hard cases, the problem could be studied in restricted graph classes like, for example, planar graphs.

Next, decompositions into different kinds of forests like star forests can be considered where a star forest is called $k$-bounded for some positive integer $k$ if each of its components is a star consisting of a central vertex and at most $k$ leaves. Observe that for $k \in \{1,2\}$, a $k$-bounded star forest is the same as a $k$-bounded linear forest. Our results hence yield that we can decide in polynomial time whether a given graph can be decomposed into a $2$-bounded star forest and a matching and that it is NP-complete to decide whether a given graph can be decomposed into two 2-bounded linear forests. To our best knowledge, the following question is open. 

\begin{Problem}
Determine the complexity of deciding whether a given graph can be decomposed into a matching and a star forest.
\end{Problem}

Further, similar questions can be asked for digraphs where a directed linear forest is a vertex-disjoint collection of directed paths. For example, the following question is open:

\begin{Problem}
Determine the complexity of deciding whether a given digraph can be decomposed into a matching and a directed linear forest.
\end{Problem}

For the result in Section \ref{gir9}, it would be interesting to see whether the constant in Theorem \ref{girth9} can be improved.

\begin{Problem}
What is the minimum integer $\alpha$ such that every planar, subcubic graph of girth at least $\alpha$ can be decomposed into a linear forest and a matching?
\end{Problem}

Observe that an improvement of this constant to 8 would imply the result of Kronk, Radlowski and Franen \cite{Kronk19743} and an improvement of this constant to 7 would imply the result of Bobuelle and Kardoš in \cite{BONDUELLE2022113002}. On the other hand, it is easy to see that $\alpha \geq 6$. Indeed, cubic, planar graphs of girth 5 are well-known to exist and a simple edge counting argument shows that they cannot even be decomposed into a matching and an arbitrary forest.
Interestingly, we are not aware of an example showing $\alpha \geq 7$, meaning a subcubic, planar graph of girth 6 that cannot be decomposed into a linear forest and a matching. 
\section*{Acknowledgement}
We wish to thank Dániel Marx for making us aware of the results in \cite{CORNU}.
\bibliographystyle{alpha}
\bibliography{sample}

\newcommand{\etalchar}[1]{$^{#1}$}
\begin{thebibliography}{KKW{\etalchar{+}}13}

\bibitem[AEH81]{Akiyama1981CoveringAP}
Jin Akiyama, Geoffrey Exoo, and Frank Harary.
\newblock Covering and packing in graphs {IV}: Linear arboricity.
\newblock {\em Networks}, 11:69--72, 1981.

\bibitem[ATW01]{atw}
Noga Alon, V.~Teague, and N.~Wormald.
\newblock Linear arboricity and linear k-arboricity of regular graphs.
\newblock {\em Graphs and Combinatorics}, 17:11--16, 01 2001.

\bibitem[BFHP84]{BERMOND1984123}
J.C. Bermond, J.L. Fouquet, M.~Habib, and B.~Peroche.
\newblock On linear k-arboricity.
\newblock {\em Discrete Mathematics}, 52(2):123--132, 1984.

\bibitem[BK22]{BONDUELLE2022113002}
Sebastien Bonduelle and František Kardoš.
\newblock Subcubic planar graphs of girth $7$ are class {I}.
\newblock {\em Discrete Mathematics}, 345(10):113002, 2022.

\bibitem[BKPY05]{conj24}
József Balogh, Martin Kochol, András Pluhár, and Xingxing Yu.
\newblock Covering planar graphs with forests.
\newblock {\em Journal of Combinatorial Theory, Series B}, 94(1):147--158,
  2005.

\bibitem[BKS03]{Berman2003ApproximationHO}
Piotr Berman, Marek Karpinski, and Alex~D. Scott.
\newblock Approximation hardness of short symmetric instances of {MAX-3SAT}.
\newblock {\em Electron. Colloquium Comput. Complex.}, TR03, 2003.

\bibitem[Cor88]{CORNU}
Gérard Cornuéjols.
\newblock General factors of graphs.
\newblock {\em Journal of Combinatorial Theory, Series B}, 45(2):185--198,
  1988.

\bibitem[CW17]{CRANSTON2017766}
Daniel~W. Cranston and Douglas~B. West.
\newblock An introduction to the discharging method via graph coloring.
\newblock {\em Discrete Mathematics}, 340(4):766--793, 2017.

\bibitem[DT92]{Hoyer}
Dorit Dor and Michael Tarsi.
\newblock Graph decomposition is {NPC} - a complete proof of {H}olyer's
  conjecture.
\newblock In {\em Proceedings of the Twenty-Fourth Annual ACM Symposium on
  Theory of Computing}, STOC '92, page 252–263, New York, NY, USA, 1992.
  Association for Computing Machinery.

\bibitem[Gon09]{GONCALVES2009314}
D.~Gonçalves.
\newblock Covering planar graphs with forests, one having bounded maximum
  degree.
\newblock {\em Journal of Combinatorial Theory, Series B}, 99(2):314--322,
  2009.

\bibitem[Hol81]{doi:10.1137/0210055}
Ian Holyer.
\newblock The {NP}-completeness of edge-coloring.
\newblock {\em SIAM Journal on Computing}, 10(4):718--720, 1981.

\bibitem[Jia18]{treespathsandspiders}
Minghui Jiang.
\newblock Trees, {P}aths, {S}tars, {C}aterpillars and {S}piders.
\newblock {\em Algorithmica}, 80:1968--1982, 2018.

\bibitem[JJJ21]{smallsubgraphcomponents}
Rain Jiang, Kai Jiang, and Minghui Jiang.
\newblock Decomposing a graph into subgraphs with small components, 2021.

\bibitem[JY17]{ndt}
Hongbi Jiang and Daqing Yang.
\newblock Decomposing a graph into forests: The {N}ine {D}ragon {T}ree
  conjecture is {T}rue.
\newblock {\em Combinatorica}, 37(6):1125--1137, Dec 2017.

\bibitem[KKW{\etalchar{+}}13]{k=1d=2}
Seog-Jin Kim, Alexandr~V. Kostochka, Douglas~B. West, Hehui Wu, and Xuding Zhu.
\newblock Decomposition of sparse graphs into forests and a graph with bounded
  degree.
\newblock {\em Journal of Graph Theory}, 74(4):369--391, 2013.

\bibitem[KRF74]{Kronk19743}
H.~Kronk, M.~Radlowski, and B.~Franen.
\newblock On the line chromatic number of triangle-free graphs.
\newblock {\em Abstract Graph Theory Newsletter}, 3:3, 1974.

\bibitem[MM22]{SNDTisolates}
Sebastian Mies and Benjamin Moore.
\newblock The strong nine dragon tree conjecture is true for $d \leq k+1$,
  2022.

\bibitem[MORZ12]{sndtck1d2}
Mickael Montassier, Patrice {Ossona de Mendez}, André Raspaud, and Xuding Zhu.
\newblock Decomposing a graph into forests.
\newblock {\em Journal of Combinatorial Theory, Series B}, 102(1):38--52, 2012.

\bibitem[NW64]{nash}
Crispin St. J.~A. Nash-Williams.
\newblock Decomposition of finite graphs into forests.
\newblock {\em Journal of the London Mathematical Society 39.1}, page~12, 1964.

\bibitem[Pé84]{PEROCHE1984195}
B.~Péroche.
\newblock {NP}-completeness of some problems of partitioning and covering in
  graphs.
\newblock {\em Discrete Applied Mathematics}, 8(2):195--208, 1984.

\bibitem[Sch78]{MNAE3SAT}
Thomas~J. Schaefer.
\newblock The complexity of satisfiability problems.
\newblock In {\em Proceedings of the Tenth Annual ACM Symposium on Theory of
  Computing}, STOC '78, page 216–226, New York, NY, USA, 1978. Association
  for Computing Machinery.

\bibitem[Tho99]{THOMASSEN1999100}
Carsten Thomassen.
\newblock Two-coloring the edges of a cubic graph such that each monochromatic
  component is a path of length at most 5.
\newblock {\em Journal of Combinatorial Theory, Series B}, 75(1):100--109,
  1999.

\bibitem[Yan18]{YANG201840}
Daqing Yang.
\newblock Decomposing a graph into forests and a matching.
\newblock {\em Journal of Combinatorial Theory, Series B}, 131:40--54, 2018.

\end{thebibliography}

\end{document}